\documentclass[aap,preprint]{imsart}

\RequirePackage{amsthm,amsmath,amsfonts,amssymb}
\RequirePackage[numbers]{natbib}
\RequirePackage[colorlinks,citecolor=blue,urlcolor=blue]{hyperref}
\RequirePackage{graphicx}
\usepackage{dsfont}

\startlocaldefs
\theoremstyle{plain}
\newtheorem{theorem}{Theorem}[section]

\newtheorem{lemme}[theorem]{Lemma}
\newtheorem{prop}[theorem]{Proposition}
\newtheorem{trick}[theorem]{Trick} 

\theoremstyle{remark}

\newtheorem{remark}[theorem]{Remark}

\newtheorem*{assumption}{Assumption}

\endlocaldefs

\def\si{\mathbf{s}}
\def\cA{\mathcal{A}}

\def\cC{\mathcal{C}}
\def\cD{\mathcal{D}}

\def\cF{\mathcal{F}}

\def\cN{\mathcal{N}}

\def\cX{\mathcal{X}}

\def \R{\mathbb{R}}

\def \1{\mathds{1}}

\newcommand{\E}{\operatorname{\rm{E}}}
\renewcommand{\P}{\operatorname{\rm{P}}}
\newcommand{\Var}{\operatorname{Var}}
\newcommand{\Cov}{\operatorname{Cov}}

\begin{document}

\begin{frontmatter}
\title{Mean number and correlation function of critical points of isotropic Gaussian fields and some results on GOE random matrices}
\runtitle{Mean number and correlation  of critical points}

\begin{aug}
\author[A]{\fnms{Jean-Marc} \snm{Aza\"is}\ead[label=e1]{jean-marc.azais@math.univ-toulouse.fr}}
\and
\author[B]{\fnms{C\'eline} \snm{Delmas}\ead[label=e2]{celine.delmas.toulouse@inrae.fr}}


\address[A]{Institut de Math\'ematiques de Toulouse,
Universit\'e Paul Sabatier,
\printead{e1}}

\address[B]{Math\'ematiques et Informatique Appliqu\'ees Toulouse UR875,
INRAE,
\printead{e2}}

\end{aug}

\begin{abstract}
Let $\mathcal{X}= \{X(t) : t \in  \mathbb{R}^N \} $ be an isotropic Gaussian random field  with real values.
In a first part we study the mean number of critical points of $\mathcal{X}$ with index $k$ using random matrices tools.
We obtain an exact expression for the probability density of the $k$th eigenvalue of a $N$-GOE matrix.
We deduce some exact expressions for the mean number of critical points with a given index. 
In a second part we study  attraction or repulsion between these critical points. A measure is the correlation function.
We prove attraction between critical points when $N>2$, neutrality for $N=2$ and repulsion for $N=1$.
The attraction between critical points that occurs when the dimension is greater than two is due to  critical points with adjacent indexes.
A strong repulsion between maxima and minima is proved. The correlation function between maxima (or minima) depends on the dimension of the ambient space.
\end{abstract}



\end{frontmatter}

\section{Introduction}\label{s:intro}
Critical points of random fields play an important role in small or large dimension.
In large dimension data they appear in the study of algorithms of maximisation of the likelihood \cite{benarous2, ros}.
In smaller dimension they play a role in many applications from various areas: detection of peaks in a random field \cite{cheng3, taylor, worsley1, worsley2},
optimisation of a response  modelled by a random field \cite{ginsbourger}, modelling of random sea waves \cite{longuet, Alo, rychlik}.
Critical points  and their indexes  are an important  element of the geometry of random fields.
They appear  in the computation of the Euler characteristic \cite{adlertaylor, cheng}.
They are a subject  of study  on their own as in  \cite{nicolaescu, benarous, BCW, cheng2}.

Let $\cX= \{X(t) : t \in  \R^N \} $ be an isotropic Gaussian random field  with real values.
If we look at the repartition of the critical points of $\cX$ as a function of their index in dimension two,
considerations of symmetry and of Euler characteristic   (see, for example, \cite{adler}, lemma 6.1.1) imply  that if $\mathcal N^c(S)$, $\mathcal N^c_{0}(S)$, $\mathcal N^c_{2}(S)$, $\mathcal N^c_{1}(S)$  are respectively
the number  of  critical, minimum, maximum and saddle points  on a  given set $S$
\begin{equation}\label{e:adler}
\E(\mathcal N^c_{0}(S)) = \E(\mathcal N^c_{2}(S)) = \frac 1 2 \E(\mathcal N^c_{1}(S)) =  \frac 1 4 \E(\mathcal N^c(S)).
\end{equation}
In higher dimensions simulations suggest that such a simple result does not hold true. 
 The purpose  of Section \ref{s:mean} is to compute these expectations using  random matrices tools.
With this objective in mind, we obtain an exact expression for the probability density of the $k$th ordered eigenvalue
of a $N$-GOE matrix (see Theorem \ref{lemmedensitevpk} and \eqref{eq:q22}--\eqref{eq:q55}). We deduce
exact expressions for the mean number of critical points with a given index (Propositions \ref{p:NPC}, \ref{prop:n=3} and \ref{prop:n=4}
). In particular, for $N=3$, if we denote by $\mathcal N^c_{0}(S)$,  $\mathcal N^c_{1}(S)$, $\mathcal N^c_{2}(S)$ and $\mathcal N^c_{3}(S)$ respectively the number of minimum,
the number of critical points with index $1$ and $2$ (see Section \ref{s:mean}, equation \eqref{nombre1} for the definition) and the number of maximum on a given set $S$, we obtain (see Proposition \ref{prop:n=3})
\begin{align*}
\E(\mathcal N^c_{0}(S))& = \E(\mathcal N^c_{3}(S)) = \frac{29-6\sqrt{6}}{116} \E(\mathcal N^c(S)),\\
\E(\mathcal N^c_{1}(S))& = \E(\mathcal N^c_{2}(S)) = \frac{29+6\sqrt{6}}{116} \E(\mathcal N^c(S)).
\end{align*}
Proposition \ref{prop:n=4} gives these expectations for $N=4$. 

On the other hand Section \ref{s:process} studies how  the critical points are spread  in the space.
In a pioneering work, Belyaev, Cammarota  and Wigman \cite{BCW} study  the attraction or repulsion  (see Section \ref{s:attractionrepulsion} for a precise definition)  of the point process  of critical points of  a particular random wave model in  dimension 2.
In the case of random processes ($N=1$),  it was generally admitted  that repulsion between  crossings or critical points occurs.
In fact this result has never been written explicitly  so it is the object of Section \ref{s:1d} with Proposition \ref{p:crit}.
With  respect to this behaviour the result of \cite{BCW} is unexpected since  no repulsion occurs between critical points.  
The object of Section \ref{s:correlation}  is to explore the phenomenon  of attraction or repulsion for a large class of random fields,   in all dimensions  and for each type of indexes.
A precise definition of attraction and repulsion is given in Section \ref{s:attractionrepulsion}.  
Theorem \ref{t1} proves attraction between critical points when $N>2$, neutrality for $N=2$ and repulsion for $N=1$.
Theorem \ref{t2} shows that the attraction between critical points that occurs when the dimension is greater than $2$ is due to the attraction between critical points with adjacent indexes.
In Theorem \ref{theo3} we prove a strong repulsion, growing with $N$, between maxima and minima. Finally Theorem \ref{theo4} gives an upper bound to the correlation function between maxima
(or equivalently minima) proving repulsion for $N<5$.

In Appendix \ref{a:1}  we  prove the validity of all the Kac-Rice formulas used in the paper. 
Appendix \ref{s:proofs} is devoted 
to the proofs of Lemmas  \ref{l:moments}, \ref{lemmeconditionaldistribution} and \ref{lemme2}.
 
\section{Notation, main assumptions and background}
\subsection{Notation}
\begin{itemize}
\item $\varphi(\cdot)$ and $\Phi(\cdot)$ are respectively the probability density and the cumulative distribution
 function of a standard Gaussian variable, $\bar{\Phi} (\cdot):=1-\Phi(\cdot)$.\\

\item If $X$ and $Y$ are random vectors,  $\Var(X)$ denotes the variance-covariance matrix of $X$ and
$$
\Cov(X,Y) := \E\Big( \big(X-\E(X)\big) \big(Y-\E(Y)\big)^\top\Big).
$$

\item For $X$ and $Y$ two random vectors $p_{X}(x)$ is the probability density of $X$ at $x$ and $p_{X}(x/Y=y)$ the probability density of $X$ conditionally to $Y=y$ when these densities exist.\\

\item For a random field  $ \cX= \{X(t) : t \in  \R^N \} $, we denote by $\nabla X(t)$ the gradient  of $X(t)$ and by $\nabla^2 X(t)$ its Hessian, when they exist.\\
 
\item $i(M)$ is the index of the matrix $M$: the number of its negative eigenvalues.\\

\item $z^+$ and $z^-$ denote respectively $\sup (0,z)$ and $-\inf(0,z)$.\\

\item $\displaystyle X_{i_1i_2\ldots i_n}(t)$ denotes $\displaystyle \frac{\partial^n X(t)}{\partial t_{i_1} \partial t_{i_2} \ldots \partial t_{i_n}}$.\\

\item $Id_n$ is the identity matrix of size $n$; $J_{n,p}$ the ($n\times p$)-matrix with all elements equal to $1$.\\
\item  $|S|$ is the Lebesgue measure of  the Borel set $S$.\\
\item  For a random field  $ \cX= \{X(t) : t \in  \R^N \} $ the number of critical points is precisely defined for every Borel set $S$ as
 \begin{equation*}
{\cN}^c(S):=\#\{t\in S : \nabla X(t)=0 \}.
\end{equation*}
\end{itemize}

\subsection{Main assumptions}
In the sequel we will use  the following assumptions (A$k$), with $ k = 2$, $3$ or $4$.

\begin{assumption}[A$k$]

 \begin{itemize}
\item $ \cX= \{X(t) : t \in  \R^N \} $ is a stationary and isotropic Gaussian field  with real values. We assume, without loss of generality, that it is centred  with variance 1. We set
$$
\E(X(s)X(t))=c(s,t)=\mathbf r(||s-t||^2).
$$

\item $\mathbf r$ is of class $\cC^k$.  This is equivalent to the existence of a finite $2k$th spectral moment $\lambda_{2k}$ and
it implies that $ \cX$ is $k$ times differentiable  in quadratic mean. Note that  $\forall \ell=1,\ldots, N$ and $n=1,\ldots, k$

\begin{align}
1&= \Var(X(t)) = \mathbf r(0), \nonumber \\
\lambda_{2n}&:=\Var\left(\frac{\partial^{n}X(t)}{\partial t_{\ell}^n}\right)= (-1)^n\frac{(2n)!}{n!}\mathbf r^{(n)}(0) . \label{spectralmoments}
\end{align}

\item To avoid the trivial case  of a constant random field, we assume that $ \lambda_2 = -2\mathbf r'(0)  >0$.\\

\item When $N>1$  we know from Lemma \ref{l:moments} below  that
$\displaystyle \lambda_{2n}\lambda_{2n-4}>\lambda_{2n-2}^2$ for $n=2,\ldots,k$. When $N=1$, we have to assume these relations to avoid the particular case of  the sine-cosine process.\\

\item When $k=2$ we have to assume, in addition for the validity of the Kac-Rice formula, that $\cX$ is $C^2$ which is  slightly stronger than the finiteness  of $\lambda_4$ .
When $k \geq 3$  the fact that  $\cX$ is $C^2$  is a consequence of  the finiteness  of $\lambda_6$.
\end{itemize}
\end{assumption}

Note that, of course,  (A4) is stronger than (A3) which is itself stronger than (A2).

\subsection{Background}
We give two lemmas that will be needed in the next sections.
The first one gives  universal properties for $\cX$. The second one is more technical.
\begin{lemme}\label{l:moments}
Let $\cX$ be a stationary and isotropic Gaussian field.
When $\lambda_{2k}$ is finite, the spectral moments satisfy for $n=2,\ldots, k$
\begin{equation}\label{eq0:lemme1}
\lambda_{2n}\lambda_{2n-4}\geq \frac{(2n-1)}{(2n-3)}\frac{(2n-4+N)}{(2n-2+N)} \lambda_{2n-2}^2.
\end{equation}
Moreover when $ i_1+j_1,\ldots,i_N+j_N$ are all even and when all the quantities hereunder are well defined
\begin{align}
\E\left( \frac{\partial^{|i|} X(t)}{\partial t_1^{i_1}, \ldots,\partial t_N^{i_N} }
  \frac {\partial^{|j|} X(t)}{\partial t_1^{j_1}, \ldots,\partial t_N^{j_N} }  \right)&=
    \left. \frac {\partial^{|i|+|j|} c(s,t)}{\partial s_1^{i_1}, \ldots,\partial s_N^{i_N}    \partial t_1^{j_1}, \ldots,\partial t_N^{j_N}   }\right|_{s=t} \nonumber\\
=  (-1)^{|j|}\mathbf r^{(|\beta|)}(0)\prod_{\ell=1}^N\frac{(2\beta_{\ell})!}{\beta_{\ell}!} 
  &=(-1)^{|\beta|+|j|}\lambda_{2|\beta|}\frac{|\beta|!}{(2|\beta|)!}\prod_{\ell=1}^N\frac{(2\beta_{\ell})!}{\beta_{\ell}!}, \label{eq2:lemme1}
\end{align}
where $\displaystyle |\beta|=\frac{|i|+|j|}{2}$ and $\displaystyle \beta_{\ell}=\frac{i_{\ell}+j_{\ell}}{2}$ for $\ell=1,\ldots,N$. In the other cases the covariances vanish.
\end{lemme}
\begin{remark}
Note that the coefficients $\displaystyle K(n,N):= \frac{(2n-1)}{(2n-3)}\frac{(2n-4+N)}{(2n-2+N)}$ are  greater than 1 when $N>1$ and take the value 1 when $N=1$.
\end{remark}
The proof of Lemma \ref{l:moments} is given in Appendix \ref{s:proofs}.\\

 In the sequel  we need  a precise description of the distribution of $X(t), \nabla X(t), \nabla^2 X(t)$
and, for technical reasons  in the last section, we need to add $X_{1ii}, X_{11\ell} , X_{1111} $ for $1\leq i \leq N$ and $2\leq \ell \leq N$.
To get independence we have to partition  this vector as follows. Set 
  \begin{align*}
\zeta_1&:=\left(X_2(t),\ldots,X_{\ell -1}(t),X_{\ell+1}(t),X_N(t)\right)^\top ,\\
\zeta_2&:=\left(X_{ij}(t), i,j=1,\ldots,N\mbox{ and }i\neq j\right)^\top ,\\
\zeta_3&:=\left(X(t),X_{1111}(t),X_{11}(t),X_{22}(t),X_{33}(t),\ldots,X_{NN}(t)\right)^\top ,\\
\zeta_4&:=\left(X_1(t),X_{111}(t),X_{122}(t),X_{133}(t),\ldots,X_{1NN}(t)\right)^\top ,\\
\zeta_5&:=\left(X_{\ell}(t),X_{11\ell}(t)\right)^\top.
\end{align*}

\begin{lemme}\label{lemme1} Under Assumption (A$4$)
the random vectors $\zeta_1$, $\zeta_2$, $\zeta_3$, $\zeta_4$ and $\zeta_5$   defined above are Gaussian, independent, centred with variance matrices respectively given by
$$
\Var \left( \zeta_1 \right)=\lambda_2 Id_{N-2},\mbox{ }
\Var \left( \zeta_2 \right)=\frac{\lambda_4}{3} Id_{N(N-1)/2}, \mbox{ }
\Var \left( \zeta_3 \right)=M_{(N+2)}, 
$$
$$
\Var \left( \zeta_4 \right)=\tilde M_{(N+1)},\mbox{ }
\Var \left(\zeta_5\right)=\left(
\begin{array}{cc}
\lambda_2 & -\lambda_4/3\\
-\lambda_4/3 & \lambda_6/5
\end{array}
\right),
$$
where    
$$
M_{(N+2)}=\left(
\begin{array}{cccc}
1 & \lambda_4 & -\lambda_2 & \begin{array}{ccc} -\lambda_2 & \hdots & -\lambda_2 \end{array}\\
\lambda_4 & \lambda_8 & -\lambda_6 & \begin{array}{ccc} -\lambda_6/5 & \hdots & -\lambda_6/5 \end{array}\\
 -\lambda_2 & -\lambda_6 & \lambda_4 & \begin{array}{ccc} \lambda_4/3 & \hdots & \lambda_4/3 \end{array}\\
\begin{array}{c} -\lambda_2 \\ \vdots \\ -\lambda_2 \end{array}& \begin{array}{c} -\lambda_6/5 \\ \vdots \\ -\lambda_6/5 \end{array} &  \begin{array}{c} \lambda_4/3 \\ \vdots \\ \lambda_4/3 \end{array} & \displaystyle \left(\frac{2\lambda_4}{3}Id_{N-1}+\frac{\lambda_4}{3}J_{N-1}\right)
\end{array}
\right)
$$
and 
$$
\tilde{M}_{(N+1)}=\left(
\begin{array}{ccc}
 \lambda_2 & -\lambda_4 & \begin{array}{ccc} -\lambda_4/3 & \hdots & -\lambda_4/3 \end{array}\\
 -\lambda_4 & \lambda_6 & \begin{array}{ccc} \lambda_6/5 & \hdots & \lambda_6/5 \end{array}\\
 \begin{array}{c} -\lambda_4/3 \\ \vdots \\ -\lambda_4/3 \end{array} &  \begin{array}{c} \lambda_6/5 \\ \vdots \\ \lambda_6/5 \end{array} & \displaystyle \left(\frac{2\lambda_6}{15}Id_{N-1}+\frac{\lambda_6}{15}J_{N-1}\right)
\end{array}
\right).
$$
Moreover we have
\begin{equation} \label{equation:determinant1}
 \det \left(\Var\left(X_{11}(t),X_{22}(t),\ldots, X_{NN}(t)\right)\right)=(N+2)2^{N-1}\left(\frac{\lambda_4}{3}\right)^N,
\end{equation}
\begin{equation} \label{equation:determinant3}
 \det\left(\Var\left(X_1(t),X_{122}(t),\ldots,X_{1NN}(t)\right)\right)=\left(\frac{2\lambda_6}{15}\right)^{N-2}\frac{3(N+1)\lambda_2\lambda_6-5(N-1)\lambda_4^2}{45}.
\end{equation}
\end{lemme}

\begin{proof}[Proof of Lemma \ref{lemme1}]
The joint distribution of $\zeta_1,\zeta_2,\zeta_3,\zeta_4$ is a direct consequence of \eqref{eq2:lemme1}. Now let us prove \eqref{equation:determinant1} and \eqref{equation:determinant3}.
We have
$$
 \Var(X_{11}(t),X_{22}(t),\ldots,X_{NN}(t))=\frac{2\lambda_4}{3}Id_N+\frac{\lambda_4}{3}J_{N,N}.
$$
For $x, y\in \R $ it is well known that
\begin{equation}
\det (xId_N + yJ_{N,N})=x^{N-1}(x+Ny).\label{eqdetxid+yj}
\end{equation}
So we obtain \eqref{equation:determinant1}. We have
$$
\det\left(\Var   \left(X_1(t),X_{122}(t),\ldots,X_{1NN}(t)\right)    \right)=\det \left[
\begin{array}{cc}
\lambda_2 & \tilde{A}_{12} \\
\tilde{A}_{21} & \tilde{A}_{22}
\end{array}
\right],
$$
with
$\tilde{A}_{12}=\left[
\begin{array}{ccc}
-\lambda_4/3 & \cdots & -\lambda_4/3 
\end{array}
\right]
$, $\tilde{A}_{21}=\tilde{A}_{12}^\top$ and $\displaystyle \tilde{A}_{22}=\frac{2\lambda_6}{15}Id_{N-1}+\frac{\lambda_6}{15}J_{N-1,N-1}$.\\
Then using the fact that for a partitioned matrix
$
\tilde{A}=\left[
\begin{array}{cc}
 \tilde{A}_{11} & \tilde{A}_{12} \\
 \tilde{A}_{21} & \tilde{A}_{22}
\end{array}
\right]
$
we have (see \cite{searle} p.46)
\begin{equation}
\det \tilde{A}=\det \tilde{A}_{11}\times \det (\tilde{A}_{22}-\tilde{A}_{21}\tilde{A}_{11}^{-1}\tilde{A}_{12}), \label{eqsearle}
\end{equation}
we obtain
\begin{multline*}
\det\left(\Var   \left(X_1(t),X_{122}(t),\ldots,X_{1NN}(t)\right)    \right)\\
=\lambda_2\det \left(\frac{2\lambda_6}{15}Id_{N-1}+J_{N-1,N-1}\left(\frac{3\lambda_2\lambda_6-5\lambda_4^2}{45\lambda_2}\right)\right).
\end{multline*}
Finally using \eqref{eqdetxid+yj} we obtain \eqref{equation:determinant3}.
\end{proof}

\section{Mean number of critical points with a given index} \label{s:mean}
In this section  $\cX$  is a random field 
  satisfying Assumption (A$2$). Then it is proved in Appendix \ref{a:1} that the sample paths are almost surely Morse and this implies that the number of critical points with index $k$  in a Borel set $S$,  ${\cN}^c_k(S)$,  is well defined.   More precisely
\begin{equation}\label{nombre1}
{\cN}^c_k(S):=\#\{t\in S : \nabla X(t)=0,\ i(\nabla^2X(t))=k \}.
\end{equation}
We define also the number of  critical points of index $k$ above the level $u$ by 
 \begin{equation}\label{nombre2}
{\cN}^c_k(u,S):=\#\{t\in S : \nabla X(t)=0,\ i(\nabla^2X(t))=k, \ X(t)>u\}.  
\end{equation}
The aim of this section is to derive explicit and exact expressions for the expectation of \eqref{nombre1} and \eqref{nombre2}.
\subsection{The general case}
By Kac-Rice formulas \eqref{eq:rice2} and \eqref{eq:rice3} and Lemma \ref{lemme1} we get
\begin{align} 
\E\left({\cN}^c_k(S)\right) =& \frac{|S|}{ ( 2\pi \lambda_2)^{N/2}} \E\left( | \det(\nabla^2 X(t) )| \1_{i(\nabla^2 X(t)) =k}  \right) . \label{e:rice2k}  \\
\E ({\cN}^c_k(u,S))=&\frac{|S|}{\lambda_2^{N/2}(2\pi)^{(N+1)/2}} \int_u^{+\infty}\exp \left(-\frac{x^2}{2}\right)  \label{eq:NXuS} \\
& \times \E \left( | \det(\nabla^2 X(t) )  \1_{i(\nabla^2 X(t)) =k}  | \  \big/ X(t) =x \right)dx.\nonumber
\end{align}
Our main tool will be  the use  of random matrices theory and more precisely the GOE distribution. We recall that $G_n$ is said to have the GOE (Gaussian Orthogonal Ensemble)
distribution if it is symmetric and all its elements are independent centred Gaussian variables satisfying $\E(G_{ii}^2)=1$ and $\E(G_{ij}^2)=\frac{1}{2}$.
The relation between GOE  matrices  and   the study  of  critical points  of stationary  isotropic Gaussian  fields  is due to the following lemma initially due to \cite{AW3}
and that can be derived from Lemma \ref{lemme1}.

\begin{lemme} \label{s:goe}
Let $ G_N$  be a size $N$ GOE matrix and  $ \Lambda$ a $\cN(0,1/2)$ random variable independent of $ G_N$. Then $\nabla^2 X(t) $ is equal in distribution  to 
\begin{equation*}
\sqrt{\frac{ 2 \lambda_4   }{3}}  \bigg( G_N -  \Lambda Id_{N} \bigg),
\end{equation*}
and under the assumption that $ \lambda_4 \geq 3 \lambda_2^2$, $\nabla^2X(t)$ conditionally to $X(t)=x$,  is distributed as
\begin{equation*}
\sqrt{\frac{ 2 \lambda_4   }{3}}  \bigg( G_N -  \tilde{\Lambda} Id_{N} \bigg),
\end{equation*}
where $\tilde{\Lambda}$ is a $\displaystyle \cN \left(\lambda_2x\sqrt{\frac{3}{2\lambda_4}},\frac{\lambda_4-3\lambda_2^2}{2\lambda_4}\right)$ random variable independent of $ G_N$.
\end{lemme}

\noindent We recall that the joint density $f_N$ of the eigenvalues $ (\mu_1, \ldots, \mu_N) $ of a $N$-GOE matrix
 (see Theorem 3.3.1 of \cite{mehta}) is given by:
\begin{equation}
f_N(\mu_1,\ldots,\mu_N)=k_N\exp\left(-\frac{\sum_{i=1}^N\mu_i^2}{2}\right)\prod_{1\leq i<k\leq N}|\mu_k-\mu_i|\ ,\label{eqdensitevp}
\end{equation}
where:
\begin{equation}
k_N:=(2\pi)^{-N/2}\left(\Gamma(3/2)\right)^N\left(\prod_{i=1}^N\Gamma(1+i/2)\right)^{-1}. \label{eqkn}
\end{equation}

\noindent Using Lemma \ref{s:goe}, \eqref{e:rice2k}, \eqref{eq:NXuS} and \eqref{eqdensitevp} we get the following proposition.
\begin{prop}\label{p:NPC}
Let  $L_p$  be the $p$th ordered eigenvalue of a $(N+1)$-GOE matrix ($L_1< L_2<\ldots <L_{N+1}$). For $\cX$ and $S$ as above, under  Assumption (A$2$) 
\begin{equation}\label{eqNck}
\E\left({\cN}^c_k(S)\right) =  \frac{|S|}{\pi^{(N+1)/2}} \left( \frac{\lambda_4}{3\lambda_2}\right)^{N/2} \frac{k_N}{k_{N+1}} \frac{1}{N+1}\E \left( \exp\left(-\frac{L^2_{k+1}}{2}\right)\right).
\end{equation}
When $\displaystyle \lambda_4-3\lambda_2^2>0$,
\begin{multline}\label{eqNcku}
\E ({\cN}^c_k(u,S)) =\frac{|S|}{\pi^{(N+1)/2}} \left( \frac{\lambda_4}{3\lambda_2}\right)^{N/2} \frac{k_N}{k_{N+1}}\frac{1}{N+1}\\
\times \E \left\{\exp\left(-\frac{L_{k+1}^2}{2}\right){\bar{\Phi}}\left[\sqrt{\frac{\lambda_4}{\lambda_4-3\lambda_2^2}}\left(u-L_{k+1}\frac{\sqrt{6}\lambda_2}{\sqrt{\lambda_4}}\right)\right]\right\},
\end{multline}
and when $\displaystyle \lambda_4-3\lambda_2^2=0$
\begin{equation}\label{eqNckup}
\E ({\cN}^c_k(u,S))=\frac{|S|}{\pi^{(N+1)/2}} \frac{k_N}{k_{N+1}}\frac{\lambda_2^{N/2}}{N+1} \E \left( \exp\left(-\frac{L_{k+1}^2}{2}\right) \1_{L_{k+1}>u/\sqrt{2}} \right).
\end{equation}
\end{prop}
\begin{remark}
Such kind of result was first obtained by \cite{benarous3} for the $p$-spin spherical spin glass model defined on the Euclidean sphere of radius $\sqrt{N}$ of $\R^N$.
This result can also be found  in  \cite{cheng2} (Proposition 3.9) written in a slightly different way. In this paper we go further:
in Theorem \ref{lemmedensitevpk} we obtain an exact expression for the probability density of the $k$th ordered eigenvalue of a $N$-GOE matrix, denoted by $q_{N}^k(l)$
(see \eqref{densiteqnk1} and \eqref{eq:q22}--\eqref{eq:q55}). We can deduce exact expressions for \eqref{eqNck}, \eqref{eqNcku} and \eqref{eqNckup} as in Propositions \ref{prop:n=3}, and \ref{prop:n=4}.
\end{remark}
\begin{remark}
As remarked in \cite{AW2}  the condition $ \lambda_4 \geq 3 \lambda_2^2$ is met  if the covariance function $\mathbf r$ is a ``Schoenberg covariance'':
it is a valid covariance function  in every  dimension. Note that more general cases have been studied by \cite{cheng2}.
\end{remark}
\begin{proof}[Proof of Proposition \ref{p:NPC}] 
Using \eqref{e:rice2k}, Lemma \ref{s:goe} and \eqref{eqdensitevp} 
\begin{align*}
 &\E \left(| \det \nabla^2 X(t) |\1_{i(\nabla^2 X(t)) =k}\right) \\
 &=\left(\frac{2\lambda_4}{3}\right)^{N/2} k_N N! \int_{\mu_1<\mu_2<\cdots <\mu_k <\lambda <\mu_{k+1}< \cdots <\mu_N}   
\prod_{1\leq i<k\leq N}|\mu_k-\mu_i|  \prod_{1\leq i\leq N}|\lambda-\mu_i| \\
&\quad  \exp\left(-\frac{\sum_{i=1}^N\mu_i^2}{2}\right) \pi^{-1/2} \exp(-\lambda^2/2)\exp(-\lambda^2/2)d\mu_1 \ldots d\mu_N d\lambda\\
 &= \left(\frac{2\lambda_4}{3}\right)^{N/2}\pi^{-1/2} \frac{k_N} {k_{N+1}}  \frac{N !} {(N+1)!}  \E \left ( \exp\left(-\frac{L_{k+1}^2}{2}\right)\right),
\end{align*}
which proves \eqref{eqNck}. Using Lemma \ref{s:goe},
by \eqref{eq:NXuS} and \eqref{eqdensitevp} we can write
\begin{multline*}
\E ({\cN}^c_k(u,S)) =
\frac{|S|}{\lambda_2^{N/2}(2\pi)^{(N+1)/2}} \int_u^{+\infty}\exp \left(-\frac{x^2}{2}\right)
\left(\frac{2\lambda_4}{3}\right)^{N/2} k_N N! \\
\times \int_{\mu_1<\mu_2<\cdots <\mu_k <y <\mu_{k+1}< \cdots <\mu_N}   
\prod_{1\leq i<k\leq N}|\mu_k-\mu_i|  \prod_{1\leq i\leq N}|y-\mu_i| 
\exp\left(-\frac{\sum_{i=1}^N\mu_i^2}{2}\right) \\
\times \exp(-y^2/2)\exp(y^2/2)p_Y(y)d\mu_1 \ldots d\mu_N \ dy\ dx.
\end{multline*}
Integrating first with respect to $x$ we obtain \eqref{eqNcku}. In the same way, when $\displaystyle \lambda_4-3\lambda_2^2=0$,
$\nabla^2X(t)/X(t)=x$ is distributed as $\displaystyle \sqrt{2\lambda_2^2}G_N-\lambda_2xId_N$ and following the same lines as previously we obtain \eqref{eqNckup}.
\end{proof}

\subsection{Exact expressions}
Before giving exact expressions for $\E ({\cN}^c_k(S))$ and $\E ({\cN}^c_k(u,S))$ we obtain, in the following paragraph \ref{paragrapheGOE}, some results concerning the probability density of
the $k$th ordered eigenvalue of a $N$-GOE matrix. These results will be used in paragraph \ref{exactexpressions}, with Propositions \ref{p:NPC},
to derive exact expressions for $\E ({\cN}^c_k(S))$ and $\E ({\cN}^c_k(u,S))$.

\subsubsection{Probability density of the $k$th ordered eigenvalue of a $N$-GOE matrix}\label{paragrapheGOE}
We denote by $L_1\leq L_2\leq \ldots \leq L_N$ the ordered eigenvalues of a $N$-GOE matrix. We denote by $q_N^k(\ell)$ the probability density of $L_k$.
In this section we obtain an expression for $q_N^k(\ell)$. We give explicit expressions of $q_N^k(\ell)$ for $N=2$, $3$, $4$, $5$ and $k=1,\ldots,N$. In particular we obtain that $q_3^2(\ell)$
is the density of a $\cN (0,1/2)$ random variable.\\
 
Let $\cF_p(n)$ be the set of the parts of $\{1,2,\ldots,n\}$ of size $p$. Let $I\in {\cal F}_{p}(n)$, let $\ell \in \R$,
we define $\displaystyle {{\cal D}}_i^I(\ell)=(-\infty,\ell)$ when $i\in I$ and $\displaystyle {{\cal D}}_i^I(\ell)=(\ell,+\infty)$ when $i\notin I$.
We denote by $\si(.)$ the sign function. 
We define the matrix ${\cA}^{\alpha}(I,\ell)$, for $\alpha=1$ or $2$, as follows.
When $n$ is even, ${\cA}^{\alpha}(I,\ell)$ is the $n\times n$ skew matrix whose elements are, $\forall i,j=1\ldots,n$:
\begin{equation}
a^{\alpha}_{i,j}(I,\ell)=\int_{{{\cal D}}_i^I(\ell)} dx \int_{{{\cal D}}_j^I(\ell)} dy \quad \si(y-x)x^{i-1}y^{j-1}(x-\ell)^{\alpha}(y-\ell)^{\alpha}\exp\left(-\frac{x^2+y^2}{2}\right). \label{def:Atilde1}
\end{equation}
When $n$ is odd , ${\cA}^{\alpha}(I,\ell)$ is the $(n+1)\times (n+1)$ skew matrix whose elements are defined by \eqref{def:Atilde1} $\forall i$, $j=1\ldots,n$ and we add the extra terms, for $i=1,\ldots,n$:
\begin{equation}
a^{\alpha}_{i,n+1}(I,\ell)=-a^{\alpha}_{n+1,i}(I,\ell)=\int_{{{\cal D}}_i^I(\ell)} x^{i-1}(x-\ell)^{\alpha}\exp \left(-\frac{x^2}{2}\right)dx \mbox{ and  }a^I_{n+1,n+1}(\ell)=0. \label{def:Atilde2}
\end{equation}\\

We are now able to state the following theorem that gives an exact expression for $q_{N}^k(\ell)$, the probability density of $L_k$, the $k$th ordered eigenvalue of a $N$-GOE matrix.

\begin{theorem}\label{lemmedensitevpk}
Using notations above, for $\ell \in \R$  and for $k=1,\ldots , N$ we have
\begin{equation}
q_{N}^k(\ell)=k_N N!(-1)^{k-1}\exp \left(-\frac{\ell^2}{2}\right) \sum_{I \in \cF_{k-1}(N-1)} {  \rm Pf }\left({\cA}^{1}(I,\ell)\right)\label{densiteqnk1}
\end{equation}
where $k_N$ is given by \eqref{eqkn} and ${  \rm Pf }\left({\cA}^{1}(I,\ell)\right)$ is the Pfaffian of the skew matrix ${\cA}^{1}(I,\ell)$ defined by \eqref{def:Atilde1} and \eqref{def:Atilde2} with $n=N-1$ and $\alpha=1$.
\end{theorem}

\begin{remark}
The probability density of the $k$th ordered eigenvalue of a $N$-GOE matrix was only known for $k=N$, the largest eigenvalue (see \cite{chiani}). Theorem \ref{lemmedensitevpk}
gives an expression for all $k=1,\ldots, N$.
\end{remark}

\begin{proof}[Proof of Theorem \ref{lemmedensitevpk}]
Let $G_{N-1}$ be a $(N-1)$-GOE matrix with eigenvalues denoted by $\mu_1,\ldots, \mu_{N-1}$. For $\ell \in \R$, we set for $k=2,\ldots, N-1$,
${\cal O}_{k}(\ell)=\{(\mu_i, i=1,\ldots, N-1) \in \R^{N-1} : \mu_1\leq \mu_2 \leq \ldots \leq \mu_{k-1} \leq \ell \leq \mu_{k}\leq \ldots \leq \mu_{N-1}\}$, with trivial adaptation when
 $k=1$ and $N$.
$L_1<L_2<\ldots <L_N$ denote the ordered eigenvalues of a $N$-GOE matrix. Using \eqref{eqdensitevp}, the probability density of $L_k$ is given by: 
\begin{align*}
q_{N}^{k}(l)=&N!\int_{{\cal O}_{k}(\ell)}f_N(\mu_1,\ldots,\mu_{N-1},\ell) d\mu_1\ldots d\mu_{N-1}\\
=&\frac{k_{N}N!}{k_{N-1}}
\int_{{\cal O}_{k}(\ell)} (-1)^{k-1} \det (G_{N-1}-\ell Id_{N-1})f_{N-1}(\mu_1,\ldots,\mu_{N-1})\\
&\times \exp\left(-\frac{\ell^2}{2}\right)d\mu_1\ldots d\mu_{N-1}.
\end{align*}
Thus $\displaystyle q_{N}^{k}(l)=\frac{k_{N}}{k_{N-1}}N(-1)^{k-1}\exp\left(-\frac{\ell^2}{2}\right)\gamma^{k-1}_{N-1,1}(\ell)$ where
\begin{multline*}
\gamma^{k-1}_{N-1,1}(\ell)=k_{N-1}(N-1)!\\
 \times \int_{{\cal O}_{k}(\ell)} \prod_{i=1}^{N-1}(\mu_i-\ell)\prod_{1\leq i<j\leq N-1}(\mu_j-\mu_i)\exp\left(-\frac{\sum_{i=1}^{N-1}\mu_i^2}{2}\right)d\mu_1\ldots d\mu_{N-1}.
\end{multline*}
We set $h_{m}(\mu_j,\ell)=\mu_j^{m-1}(\mu_j-\ell)\exp\left(-\mu_j^2/2\right)$.
We denote by ${\bf{H}}({\pmb{\mu}},\ell)$ the matrix $\{h_m(\mu_j,\ell),m,j=1,\ldots,N-1\}$.
Then calculating $\det({\bf {H}}({\pmb{\mu}},\ell))$ using a Vandermonde determinant we obtain:
\begin{equation*}
\gamma^{k-1}_{N-1,1}(\ell)=k_{N-1}(N-1)!\int_{{\cal O}_{k}(\ell)}\det({\bf{H}}({\pmb{\mu}},\ell)) d\mu_1\cdots \mu_{N-1}.
\end{equation*}
We denote by ${\bf{I}}({\pmb{\mu}},\ell,k)$ the $(N-1)\times (N-1)$ diagonal matrix defined by
$$
{\bf{I}}(\pmb{\mu},\ell,k)=\mbox{diag}(\1_{\mu_1\leq \ell},\ldots,\1_{\mu_{k-1}\leq \ell},\1_{\mu_k\geq \ell},\ldots,\1_{\mu_{N-1}\leq \ell})
$$
with trivial adaptations when $k=1$ or $k=N$. We set
$$
{\bf{H}}^{\star}({\pmb{\mu}},\ell,k):={\bf{H}}({\pmb{\mu}},\ell){\bf{I}}({\pmb{\mu}},\ell,k).
$$
Then
\begin{equation*}
\gamma^{k-1}_{N-1,1}(\ell)=k_{N-1}(N-1)!\int_{\mu_1\leq \ldots \leq \mu_{N-1}}\det({\bf{H}}^{\star}({\pmb{\mu}},\ell,k)) d\mu_1\cdots \mu_{N-1}.
\end{equation*}
Now we need to introduce some more notation.
Let $I$ and $J$ $\in \cF_p(n)$ and $M$ a $n\times n$ matrix, we set $\Delta_{I,J}(M)$ for $\det (M_{I,J})$. We use the convention $\Delta_{\emptyset,\emptyset}(M)=1$.
$\bar{K}=\{1,\ldots,n\}\backslash K$ for any subset $K$. 
Let $I\in {\cal F}_{p}(n)$, $\sigma_I$ is the permutation of $(1,2,\ldots,n)$ such that $\left.\sigma_I\right|_{(1,2,\ldots,p)}$ (resp. $\left.\sigma_I\right|_{(p+1,\ldots,n)}$)
is an increasing one-to-one mapping from $(1,2,\ldots,p)$ on $I$ (resp. from $(p+1,\ldots,n)$ on $\bar{I}$).
Finally we denote by $\epsilon(\sigma)$ the signature of the permutation $\sigma$ of $(1,2,\ldots,n)$.\\

\noindent For a $(n+p)\times (n+p)$ matrix $D$ we have, for any $J\in  \cF_{p}(n+p)$ fixed,
\begin{equation}
\det D=\sum_{I\in \cF_p(n+p)}\epsilon(\sigma_I)\epsilon(\sigma_J)\Delta_{I,J}(D)\Delta_{\bar{I},\bar{J}}(D).\label{calculdet}
\end{equation}
We apply it for $p=k-1$, $n=N-k$ and $J=\{1,\ldots,k-1\}$. We get $\epsilon(\sigma_J)=1$ and
\begin{multline*}
\gamma^{k-1}_{N-1,1}(\ell)=k_{N-1}(N-1)!\sum_{I\in \cF_{k-1}(N-1)}\epsilon(\sigma_I)\\
\times \int_{\mu_1\leq \ldots \leq \mu_{N-1}}\Delta_{I,J}({\bf{H}}^{\star}({\pmb{\mu}},\ell,k))\Delta_{\bar{I},\bar{J}}({\bf{H}}^{\star}({\pmb{\mu}},\ell,k))d\mu_1\cdots \mu_{N-1}.
\end{multline*}
We denote by ${\bf{H}}^I({\pmb{\mu}},\ell)$ the matrix $\{h_m(\mu_j,\ell)\1_{\mu_j\in \cD_m^I(\ell)},m,j=1,\ldots,N-1\}$.
We have
$$
\Delta_{I,J}({\bf{H}}^{\star}({\pmb{\mu}},\ell))=\Delta_{I,J}({\bf{H}}^{I}({\pmb{\mu}},\ell))\mbox{ and }\Delta_{\bar{I},\bar{J}}({\bf{H}}^{\star}({\pmb{\mu}},\ell))=\Delta_{\bar{I},\bar{J}}({\bf{H}}^{I}({\pmb{\mu}},\ell)).
$$
To check this, note that, for example, for $j\leq k-1$, the indicator function appearing in the entry $(m,j)$ of $\displaystyle {\bf{H}}^{\star}({\pmb{\mu}},\ell,k)$ is $\displaystyle \1_{\mu_j\leq \ell}$.
For every $I$ and for $m\in I$, this quantity is also equal to $\displaystyle \1_{\mu_j \in \cD_m^I(\ell)}$.\\

\noindent Then using \eqref{calculdet} with $p=k-1$, $n=N-k$ and $J=\{1,\ldots,k-1\}$, we have $\Delta_{\tilde{I},J}({\bf{H}}^I({\pmb{\mu}},\ell))=0$ when $\tilde{I}\neq I$ and
$$
\det ({\bf{H}}^{I}({\pmb{\mu}},\ell))\1_{\mu_1\leq \ldots \leq \mu_{N-1}}=\epsilon(\sigma_I)\Delta_{I,J}({\bf{H}}^{I}({\pmb{\mu}},\ell))\Delta_{\bar{I},\bar{J}}({\bf{H}}^{I}({\pmb{\mu}},\ell))\1_{\mu_1\leq \ldots \leq \mu_{N-1}}.
$$
We conclude that
$$
\gamma^{k-1}_{N-1,1}(\ell)=k_{N-1}(N-1)!\sum_{I\in \cF_{k-1}(N-1)}\int_{\mu_1\leq \ldots \leq \mu_{N-1}}\det ({\bf{H}}^{I}({\pmb{\mu}},\ell))d\mu_1\cdots \mu_{N-1}.
$$
Since $\displaystyle { \rm Pf }\left({\cA}^{1}(I,\ell)\right)=\int_{\mu_1\leq \ldots \leq \mu_{N-1}}\det ({\bf{H}}^{I}({\pmb{\mu}},\ell))d\mu_1\cdots \mu_{N-1}$ (see
\cite{bruijn}, \cite{mehta1} or equation (13) in \cite{chiani}) we obtain \eqref{densiteqnk1}. This concludes the proof.

\end{proof}

\noindent The major drawback of the result above is its complicated form. However, for small values of $N$,
we are able  to get an explicit expression for $q_{N}^k(l)$ and consequently (using \eqref{eqNck}, \eqref{eqNcku} and \eqref{eqNckup}) for $E\left({\cN}^c_k(S)\right)$ and $E\left({\cN}^c_k(u,S)\right)$.
We give some examples below and we derive Propositions \ref{prop:n=3} and \ref{prop:n=4} which are new results.\\
\\
\noindent {\bf{Examples}}: After tedious calculations we obtain\\
\\
\noindent 1. For $N=2$: $\displaystyle q_2^1(l)=q_2^2(-l)$ and
\begin{equation}
q_2^2(l)=\frac{\exp\left(-\frac{l^2}{2}\right)}{2 \sqrt{\pi}}\left[\exp\left(-\frac{l^2}{2}\right)+\sqrt{2\pi} l \Phi(l)\right]. \label{eq:q22}
\end{equation}
\noindent 2. For $N=3$: $\displaystyle q_3^1(l)=q_3^3(-l)$, $\displaystyle q_3^2(l)=\frac{\exp\left(-l^2\right)}{\sqrt{\pi}}$ and
\begin{equation}
q_3^3(l)=\frac{\exp\left(-\frac{l^2}{2}\right)}{\pi \sqrt{2}}\left[\sqrt{\pi}(2l^2-1)\Phi \left(l\sqrt{2}\right)+\sqrt{2\pi}\exp\left(-\frac{l^2}{2}\right)\Phi (l)+l\exp\left(-l^2\right)\right]. \label{eq:q33}
\end{equation}
\noindent 3. For $N=4$: $\displaystyle q_4^1(l)=q_4^4(-l)$, $\displaystyle q_4^2(l)=q_4^3(-l)$ and
\begin{multline}
q_4^3(l)= \frac{\exp\left(-\frac{l^2}{2}\right)}{2\pi} \left[\frac{3l}{2}\exp \left(-\frac{3l^2}{2}\right) +\sqrt{2\pi}\left(1-\frac{l^2}{2}\right)\bar{\Phi} (l)\exp \left(-l^2 \right)\right.\\
\left. -\frac{\pi (2l^3-3l)}{\sqrt{2}}\Phi \left(l\sqrt{2}\right) \bar{\Phi} (l)
   +\frac{3\sqrt{\pi}(1+2l^2)}{2}\Phi \left (l\sqrt{2}\right) \exp \left(-\frac{l^2}{2}\right) \right] .\label{eq:q43}
\end{multline}
\begin{multline}
q_4^4(l)= \frac{\exp\left(-\frac{l^2}{2}\right)}{2\pi} \left[\frac{3l}{2}\exp \left(-\frac{3l^2}{2}\right) -\sqrt{2\pi}\left(1-\frac{l^2}{2}\right)\Phi(l)\exp \left(-l^2 \right)\right.\\
\left. +\frac{\pi (2l^3-3l)}{\sqrt{2}}\Phi \left(l\sqrt{2}\right) \Phi(l)
   +\frac{3\sqrt{\pi}(1+2l^2)}{2}\Phi \left (l\sqrt{2}\right) \exp \left(-\frac{l^2}{2}\right) \right]. \label{eq:q44}
\end{multline}
\noindent 4. For $N=5$: $\displaystyle q_5^1(l)=q_5^5(-l)$, $\displaystyle q_5^2(l)=q_5^4(-l)$ and
\begin{multline}
q_5^3(l)=q_5^4(l)-2q_5^5(l)+\frac{\sqrt{2}\exp\left(-\frac{l^2}{2}\right)}{3 \pi^{3/2}} \left[\pi\left(4l^4-12l^2+3\right)\Phi\left(l\sqrt{2}\right)\right.\\
\left. +\sqrt{\pi}\left(3l^3-\frac{13l}{2}\right) \exp \left(-l^2\right)
 +\sqrt{2}\pi \left(l^4+3l^2+\frac{3}{4}\right)\exp \left(-\frac{l^2}{2}\right)\Phi(l) \right].\label{eq:q53}
\end{multline}
\begin{multline}
q_5^4(l)=\frac{\sqrt{2}\exp\left(-\frac{l^2}{2}\right)}{3 \pi^{3/2}} \left[\sqrt{2\pi}\exp \left(-\frac{3l^2}{2}\right)\left(\frac{l^3}{2}+\frac{5l}{4}\right)\right.\\
\left. +\sqrt{2}\pi \Phi \left(l\sqrt{2}\right)\exp \left(-\frac{l^2}{2}\right)\left(l^4+3l^2+\frac{3}{4}\right)\right].\label{eq:q54}
\end{multline}
\begin{multline}
q_5^5(l)= \frac{\sqrt{2}\exp\left(-\frac{l^2}{2}\right)}{3 \pi^{3/2}} \left[\left( 2l^4-6l^2+\frac{3}{2}\right)\pi \Phi ^2 \left (l\sqrt{2}\right)\right.\\
\left. + \left(l^4+3l^2+\frac{3}{4} \right) \sqrt{2} \pi \Phi \left(l\sqrt{2}\right) \Phi(l) \exp\left(-\frac{l^2}{2}\right) +\sqrt{2\pi}\left(\frac{l^3}{2}+ \frac{5l}{4}\right)\Phi(l)\exp \left(-\frac{3l^2}{2} \right) \right.\\
\left. +\left( 3l^3-\frac{13l}{2}\right)\sqrt{\pi}\Phi \left (l\sqrt{2}\right) \exp \left(-l^2 \right)+\left(l^2-2\right)\exp \left(-2l^2 \right)  \right]. \label{eq:q55}
\end{multline}
These probability densities are plotted in Figure \ref{figure2}. They all seem very close to a Gaussian density. But only one, $q_3^2(l)$, is exactly Gaussian. 
\begin{figure}[h!]
\begin{center}
\includegraphics[height=6cm,width=6cm]{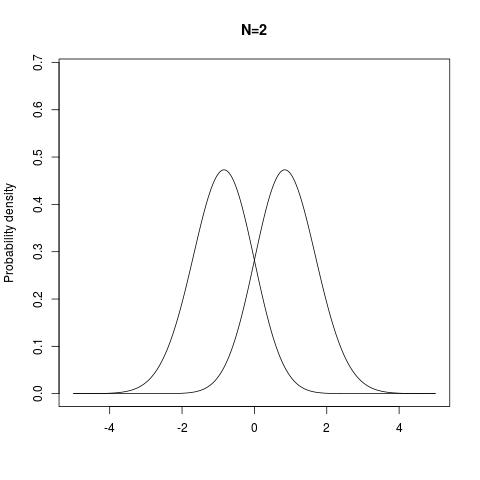}
\includegraphics[height=6cm,width=6cm]{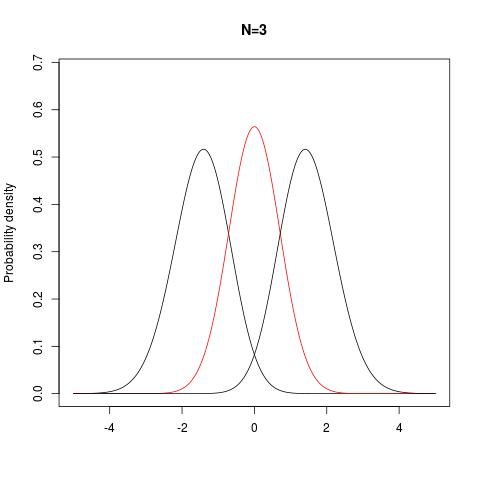} \\
\includegraphics[height=6cm,width=6cm]{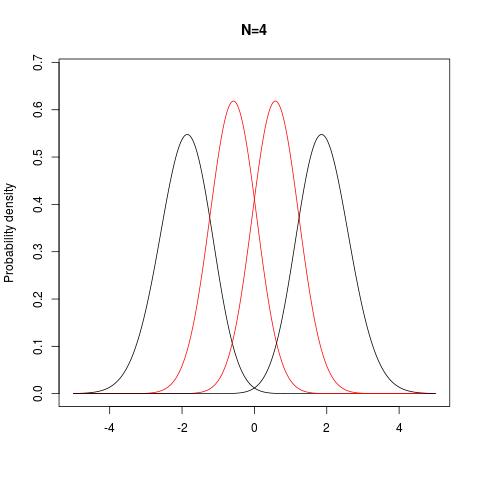}
\includegraphics[height=6cm,width=6cm]{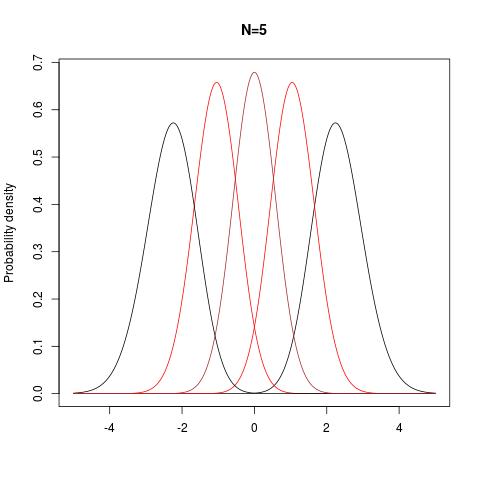}
\caption{Probability densities $q_N^k(l)$, $k=1,\ldots,N$, of the ordered  eigenvalues of a $N$ GOE matrix for $N=2,3,4,5$. }
\label{figure2}
\end{center}
\end{figure}

\subsubsection{Mean number of critical points}\label{exactexpressions}

\noindent For $N=2$, using \eqref{eqNck} and \eqref{eq:q33}, after some calculations,  we retreive \eqref{e:adler} with 
$$\E\left({\cN}^c(S)\right)= \frac{2|S|}{\sqrt{3}\pi}\left(\frac{\lambda_4}{3\lambda_2}\right).
$$
 In Proposition \ref{prop:n=3} we give the exact expression of $\E\left({\cN}^c_k(S)\right)$ when $N=3$ for $k=0,1,2,3$.
This proposition is a new result.

\begin{prop} \label{prop:n=3}
Under the conditions of Proposition \ref{p:NPC} when $N=3$
\begin{align*}
\E\left({\cN}^c_0(S)\right)=\E\left({\cN}^c_3(S)\right)&=\frac{|S|}{\pi^2 \sqrt{2}}\left(\frac{\lambda_4}{3\lambda_2}\right)^{3/2}\left(\frac{29-6\sqrt{6}}{12\sqrt{6}}\right)\\
\E\left({\cN}^c_1(S)\right)=\E\left({\cN}^c_2(S)\right)&=\frac{|S|}{\pi^2 \sqrt{2}}\left(\frac{\lambda_4}{3\lambda_2}\right)^{3/2}\left(\frac{29+6\sqrt{6}}{12\sqrt{6}}\right).
\end{align*}
Consequently:
\begin{align*}
\displaystyle \frac{\E\left({\cN}^c_0(S)\right)}{\E\left({\cN}^c(S)\right)}=\frac{\E\left({\cN}^c_3(S)\right)}{\E\left({\cN}^c(S)\right)}=\frac{29-6\sqrt{6}}{116}\simeq 0.1233 \\
\displaystyle \frac{\E\left({\cN}^c_1(S)\right)}{\E\left({\cN}^c(S)\right)}=\frac{\E\left({\cN}^c_2(S)\right)}{\E\left({\cN}^c(S)\right)}=\frac{29+6\sqrt{6}}{116}\simeq 0.3767.
\end{align*}

\end{prop}

\noindent In the same way for $N=4$, we obtain the expressions given in Proposition \ref{prop:n=4}. 

\begin{prop} \label{prop:n=4}
Set ${\cal{I}}:=\E\left(\Phi(Y)\Phi(\sqrt{2}Y)\right)$, $Y$ being  a Gaussian centred variable with variance $1/3$. Under the conditions of Proposition \ref{p:NPC} when $N=4$
\begin{align*}
\E\left({\cN}^c_1(S)\right)=\E\left({\cN}^c_3(S)\right)&=\frac{|S|}{\pi^2}\left(\frac{\lambda_4}{3\lambda_2}\right)^{2}\left(\frac{25}{24\sqrt{3}}\right),\\
\E\left({\cN}^c_0(S)\right)=\E\left({\cN}^c_4(S)\right)&=\frac{|S|}{\pi^2}\left(\frac{\lambda_4}{3\lambda_2}\right)^{2}\left(\frac{{\cal{I}}\times 100\pi - 57}{48\sqrt{3}\pi}\right),\\
\E\left({\cN}^c_2(S)\right)&=\frac{|S|}{\pi^2}\left(\frac{\lambda_4}{3\lambda_2}\right)^{2}\left(\frac{50\pi(1-2{\cal{I}}) + 57}{24\sqrt{3}\pi}\right).
\end{align*}
Consequently,
$$
\begin{array}{ccc}
\displaystyle \frac{\E\left({\cN}^c_1(S)\right)}{\E\left({\cN}^c(S)\right)}=\frac{\E\left({\cN}^c_3(S)\right)}{\E\left({\cN}^c(S)\right)}=\frac{1}{4} & \mbox{,} & 
\displaystyle \frac{\E\left({\cN}^c_0(S)\right)}{\E\left({\cN}^c(S)\right)}=\frac{\E\left({\cN}^c_4(S)\right)}{\E\left({\cN}^c(S)\right)}=\frac{{\cal{I}}\times 100\pi-57}{200\pi}\simeq 0.060 \mbox{,}
\end{array}
$$
$$
\displaystyle \frac{\E\left({\cN}^c_2(S)\right)}{\E\left({\cN}^c(S)\right)}=\frac{50\pi(1-2{\cal{I}})+57}{100\pi}\simeq 0.380.
$$
\end{prop}



 
\noindent Note that, in the same way, it is possible to obtain exact expressions for the mean number of critical points above a level $u$ for $N=2$, $3$ or $4$ using \eqref{eq:q33}, \eqref{eq:q43}, \eqref{eq:q44} and \eqref{eq:q53}, \eqref{eq:q54}, \eqref{eq:q55}
with Proposition \ref{p:NPC}.

 \section {Correlation function between critical points}\label{s:process}
In this section $\cX$ is a random field satisfying Assumption (A$2$).
 \subsection{Correlation function, two points function} \label{s:a:def}
 The correlation function of a point  process ${\cal P}$ is defined  by   \cite{berry}
   \begin{equation*}
  A(s,t) := \lim_{\epsilon \to 0}  \frac 1 {V^2(\epsilon)} \E \left({\cal P}(B(s,\epsilon)) {\cal P}(B(t,\epsilon)) \right) \mbox{ when it exists},
\end{equation*}
  where ${\cal P}(B(s,\epsilon))$ is  the number of points in the ball with center $s$ and radius $\epsilon$ and $V(\epsilon)$  its volume of this ball. 
When  the  considered process is  stationary and isotropic, this function depends only on the norm $ \rho:= \|s-t\|$ and by a small abuse of notation we will denote it by  $ A(\rho)$.\\
Suppose now that the point process ${\cal P}$  is  the process of critical points  of $\cX$. Under  our conditions, the Kac-Rice formula of order two in sufficiently small sets is valid (see Appendix \ref{a:1},
Proposition \ref{p:riceN2}, equation \eqref{eq:rice4}) . In particular 
if $S$ and $T$  are disjoint, sufficiently close and sufficiently small (as $B(s,\epsilon)$  and $B(t,\epsilon)$ for $s \neq t$  and $\|s-t\|$ and $\epsilon$ sufficiently small), the  Kac-Rice formula  of order two \eqref{eq:rice4} yields 
\begin{multline}\label{e:riceordre2}
\E[{\cN}^c(S){\cN}^c(T)] \\
=\int_{S\times T} \E\left( | \det(\nabla^2 X(s)) \det(\nabla^2 X(t)) | /\nabla X(s) = \nabla X(t) = 0\right) p_{\nabla X(s),\nabla X(t)}(0,0) ds dt,
\end{multline} 
proving that  for $ \rho $ sufficiently small, $A(\rho) $ is well defined  and given by
\begin{equation}\label{e:cor}
 A(\rho) =  \E\left( | \det(\nabla^2 X(0)) \det(\nabla^2 X(t)) | /\nabla X(0) = \nabla X(t) = 0\right) p_{\nabla X(0),\nabla X(t)}(0,0)
 \end{equation}
 with, for example, $t =\rho e_1$ and $ e_1 := (1,0,\ldots,0)$.

\noindent Some papers, as \cite{BCW} consider  the behaviour  as $ \rho \to 0$  of $$
  T(\rho)  :=\E(\cN^c (B_\rho) (\cN^c (B_\rho) -1)), \quad  B_\rho \mbox{ is any ball with radius } \rho.  
$$
It is elementary  to see that  if $ A(\rho) \simeq C \rho^d$ then $T(\rho)  \simeq C V_N^2 \rho^{d+ 2N}$ , where $ \simeq$   means, as in the rest of the paper, equivalence as $\rho \to 0$ and $V_N$ is the volume of the unit ball. 

 \subsection{Attraction, neutrality, repulsion}\label{s:attractionrepulsion}
 \begin{itemize}
 \item
   The reference of neutral point process  is the  Poisson process for which  the correlation function $A(\rho)$  is constant and 
 $T(\rho)$ behaves as $\rho^{2N}$.
 
 \item  The repulsion  is defined by the fact that  the correlation function  $A(\rho)$ tends to zero as $\rho \to 0$. Equivalently 
 $T(\rho) = o(\rho^{2N})$. Note that determinental processes \cite{hough} are a way of constructing  repulsive point processes. 
 
\item  The attraction is just the contrary: as $\rho \to 0$, $A(\rho) \to +\infty$, $\displaystyle  \frac{T(\rho) }{\rho^{2N}} \to +\infty$. 

 \end{itemize}
 As already mentioned in the preceding section, under our assumptions  the sample paths are  a.s. Morse, and the index of each critical  point can a.s. be defined. We can generalise  the definition of the correlation function by 
    \begin{equation} \label{e:aij}
  A^{i_1,i_2}(s,t) := \lim_{\epsilon \to 0}  \frac 1 {   V^2(\epsilon)} \E \big(  \cN^c_{i_1}(B(s,\epsilon)) \cN^c_{i_2}(B(t,\epsilon)) \big), \mbox{ when it exists}.
\end{equation}
The Kac-Rice formula \eqref{eq:rice5}  yields that for $\rho$ sufficiently small this  function is well defined and 
 \begin{multline}
A^{i_1,i_2} (\rho) = p_{\nabla X(0),\nabla X(t)}(0,0)  \\
\times \E\left( | \det(\nabla^2 X(0))  \1_{i(\nabla^2 X(0)) =i_1}\det(\nabla^2 X(t))\1_{i(\nabla^2 X(t)) =i_2} | / \nabla X(0) = \nabla X(t) = 0\right) , \label{eqfctcorr} 
\end{multline}
with again $t =\rho e_1$.  
   In the same way, we can consider attraction, neutrality or repulsion between critical points with indexes $i_1$ and $i_2$.\\

Before giving our main results concerning correlation functions between critical points, we give in the following paragraph some results concerning the conditional distribution
of $\nabla^2X(0), \nabla^2X(t)$ given $\nabla X(0)=\nabla X(t)=0$.

\subsection{Conditional distribution of $\nabla^2 X(0), \nabla^2 X(t)$}

In this section, for short, $\mathbf r(\rho^2)$, $\mathbf r'(\rho^2)$, $\mathbf r''(\rho^2)$, $\mathbf r'''(\rho^2)$ and $\mathbf r^{(4)}(\rho^2)$ are denoted by $\mathbf r$, $\mathbf r'$, $\mathbf r''$, $\mathbf r'''$ and $\mathbf r^{(4)}$.
We recall that $t=\rho e_1$.

\begin{lemme}\label{lemmeconditionaldistribution}
Let $\xi(0)$ and $\xi(t)$ be a representation of the distribution of $\nabla^2 X(0)$ and $\nabla^2 X(t)$ given $\nabla X(0)=\nabla X(t)=0$.
We note ${\xi_{d}}(t)$ the vector $\left(\xi_{11}(t),\ldots,\xi_{NN}(t)\right)$ and ${\xi_{u}}(t)$ the vector $\left(\xi_{12}(t), \xi_{13}(t), \ldots,\xi_{(N-1)N}(t)\right)$.
The joint distribution of
$\left({\xi_{d}}(0),{\xi_{u}}(0),{\xi_{d}}(t),{\xi_{u}}(t)\right)$ is centred Gaussian with variance-covariance matrix:
$$
\left[
\begin{array}{cccc}
 \Gamma_1 & 0 & \Gamma_3 & 0 \\
 0 & \Gamma_2 & 0 & \Gamma_4 \\
 \Gamma_3 & 0 & \Gamma_1 & 0\\
 0 & \Gamma_4 & 0 & \Gamma_2
\end{array}
\right].
$$
$\Gamma_1$ is the $N\times N$ matrix:
$$
\Gamma_1=\left[
\begin{array}{cccc}
12 \mathbf r ''(0)&  4 \mathbf r ''(0) & \cdots &4 \mathbf r ''(0) \\
4 \mathbf r ''(0) & \ddots & 4 \mathbf r ''(0) & \vdots \\
\vdots & 4 \mathbf r ''(0)  & \ddots & 4 \mathbf r ''(0) \\
4 \mathbf r ''(0) & \cdots &4 \mathbf r ''(0) & 12 \mathbf r ''(0)\end{array}
\right]
+ \frac{  \rho^2 \mathbf r '(0)}{ 2[ \mathbf r '(0)^2-(\mathbf r ' + 2 \mathbf r '' \rho^2)^2]}\times  M,
$$
with
$$
M =\left[
\begin{array}{cccc}
(12 \mathbf r '' +8 \rho^2\mathbf r ''')^2 & 4\mathbf r ''(12 \mathbf r '' +8 \rho^2\mathbf r ''')  & \cdots & 4\mathbf r ''(12 \mathbf r '' +8 \rho^2\mathbf r ''')\\
4\mathbf r ''(12 \mathbf r '' +8 \rho^2\mathbf r ''') & 16\mathbf r ''^2 &  \cdots &16\mathbf r ''^2  \\
\vdots & \vdots  & \ddots & \vdots \\
4 \mathbf r ''(12 \mathbf r '' +8 \rho^2\mathbf r ''') & 16\mathbf r ''^2 & \cdots  & 16\mathbf r ''^2 
\end{array}
\right].
$$
$\Gamma_2$ is the $\frac{N(N-1)}{2}\times \frac{N(N-1)}{2}$ diagonal matrix:
$$
\Gamma_2=\left[
\begin{array}{cc}
D_1 & 0 \\
0 & D_2
\end{array}
\right] \mbox{ with }D_1=
\left(4\mathbf r ''(0)+\frac{ 8\rho^2  (\mathbf r '')^2 \mathbf r '(0)}{\mathbf r '(0)^2 -(\mathbf r ')^2}\right) Id_{N-1}\mbox{ and }D_2=\left(4 \mathbf r ''(0)\right)Id_{\frac{(N-1)(N-2)}{2}}.
$$
We set $a:=4 \mathbf r ''+8\rho^2 \mathbf r '''$ and $d:=12 \mathbf r ''+48\rho^2 \mathbf r '''+16\rho^4 \mathbf r^{(4)}$. $\Gamma_3$ is the $N\times N$ matrix:
$$
\Gamma_3=\left[
\begin{array}{ccccc}
d & a  & \cdots & \cdots & a\\
a & 12 \mathbf r ''& 4\mathbf r ''& \cdots & 4\mathbf r ''\\
\vdots & 4\mathbf r '' & \ddots & 4\mathbf r ''& \vdots \\
\vdots & \vdots & 4\mathbf r ''& \ddots & 4\mathbf r ''\\
a & 4\mathbf r ''& \cdots & 4\mathbf r ''& 12\mathbf r ''
\end{array}
\right]
+ \frac{\rho^2(\mathbf r ' +2 \mathbf r ''\rho^2)}{ 2\big[\mathbf r '(0)^2-( \mathbf r ' +2 \mathbf r ''\rho^2)^2\big]}          \times   M .
$$
$\Gamma_4$ is the $\frac{N(N-1)}{2}\times \frac{N(N-1)}{2}$ diagonal matrix:
$$
\Gamma_4=\left[
\begin{array}{cc}
\tilde{D}_1 & 0 \\
0 & \tilde{D}_2
\end{array}
\right] \mbox{ with }\tilde{D}_1=\left(a+
\frac{ 8\rho^2  (\mathbf r '')^2 \mathbf r ' }{\mathbf r '(0)^2-\mathbf r '^2}  \right) Id_{N-1}\mbox{ and }\tilde{D}_2= 4 \mathbf r '' Id_{\frac{(N-1)(N-2)}{2}}.
$$
\end{lemme} 
 The proof of Lemma \ref{lemmeconditionaldistribution} is given in Appendix \ref{s:proofs}.\\

In the following lemma, we give the equivalent of the variance-covariance matrix of  $\nabla^2 X(0), \nabla^2 X(t)$ given $\nabla X(0)=\nabla X(t)=0$ as $\rho \rightarrow 0$.

 \begin{lemme}\label{lemme2}
Using the same notation as in Lemma \ref{lemmeconditionaldistribution},
for $j\neq k$ and $j,k\neq 1$, as $\rho \rightarrow 0$
\begin{equation}\label{eqdetvar}
\det\left(\Var \left(\nabla X(0),\nabla X(t)\right)\right) \simeq \rho^{2N} \frac{\lambda_2^N \lambda_4^N}{3^{N-1}},
\end{equation}
\begin{align}  \label{eqvarxi1j}
\Var\left(\xi_{11}(t)\right) &\simeq \frac{\rho^2}{4}\frac{(\lambda_2\lambda_6-\lambda_4^2)}{\lambda_2},\ \Var \left(\xi_{1j}(t)\right) \simeq \frac{\rho^2}{4}\frac{(9\lambda_2\lambda_6-5\lambda_4^2)}{45\lambda_2},\\
\Var\left(\xi_{jj}(t)\right) &\simeq \frac{8\lambda_4}{9},\mbox{  }\Var\left(\xi_{jk}(t)\right) \simeq \frac{\lambda_4}{3}, \label{eqvar2}
\end{align}
\begin{equation}   \label{eqvarxijk}
\Cov\left(\xi_{11}(t),\xi_{jj}(t)\right)  =  \rho^2\frac{11\lambda_2\lambda_6-15\lambda_4^2}{180\lambda_2} + o(\rho^2), \ 
\Cov\left(\xi_{jj}(t),\xi_{kk}(t)\right) \simeq \frac{2\lambda_4}{9}.
\end{equation}
\begin{equation} \label{zero1}
\mbox{All the other covariances }\Cov(\xi_{il}(t),\xi_{mn}(t)) \mbox{ are zero, }\forall i,l,m,n \in \{1,\ldots,N\}.
\end{equation}
We have of course the same relations for $\xi(0)$. \medskip
\\
\noindent Moreover we have $\forall j,k \in \{2,\ldots,N\}$ and $\forall i \in \{1,\ldots,N\}$; as $\rho \rightarrow 0$,
\begin{equation}
\Cov\left(\xi_{jk}(0),\xi_{jk}(t)\right)\simeq \Var\left(\xi_{jk}(t)\right),\mbox{  }\Cov\left(\xi_{1i}(0),\xi_{1i}(t)\right)\simeq -\Var\left(\xi_{1i}(t)\right), \label{eqcov2} 
\end{equation}
\begin{equation}
\Cov\left(\xi_{11}(0),\xi_{jj}(t)\right) = \rho^2\frac{15\lambda_4^2-7\lambda_2\lambda_6}{180\lambda_2} + o(\rho^2),\mbox{ }\Cov\left(\xi_{jj}(0),\xi_{kk}(t)\right) \simeq \frac{2\lambda_4}{9}\mbox{ for }j\neq k, \label{eqvarxi11}
\end{equation}
\begin{equation} \label{zero}
\mbox {All the other covariances }\Cov(\xi_{il}(0), \xi_{mn}(t)) \mbox{ are zero, }
\forall i,l,m,n \in \{1,\ldots,N\}.
\end{equation}
Finally we also have
\begin{equation}\label{determinant1}
\det\left(\Var\left(\xi_{11}(t),\xi_{11}(0)\right)\right)\simeq \frac{\rho^6}{144}\frac{(\lambda_4\lambda_8-\lambda_6^2)(\lambda_2\lambda_6-\lambda_4^2)}{\lambda_2\lambda_4}. 
\end{equation}
\end{lemme}

 The proof of Lemma \ref{lemme2} is given in Appendix \ref{s:proofs}. It is based  on some lengthy computations  but uses  a trick that simplify  drastically computations. We present it here in the example of the proof of \eqref{eqdetvar}.

\begin{trick}\label{trick1}
Because a determinant is invariant by adding to some row (resp. column) a linear combination of other rows (resp. columns) as $\rho \to 0$,
\begin{align*}
\det&\left(\Var\left(\nabla X(0),\nabla X(t)\right)\right)=\det (\Var(X'_1(0),\ldots,X'_N(0),X'_1(t),\ldots,X'_N(t)))\\
&= \det (\Var(X'_1(0),\ldots,X'_N(0),X'_1(t)-X'_{1}(0),\ldots,X'_N(t)-X'_{N}(0)))\\
&\simeq \rho^{2N} \det (\Var(X'_1(0),\ldots,X'_N(0),X^{''}_{11}(0),\ldots,X^{''}_{1N}(0))).
\end{align*}
\end{trick}
Using Lemma \ref{lemme1} we obtain \eqref{eqdetvar}.

 \subsection{Correlation function between critical points for random processes}    \label{s:1d}

The result below appears in a hidden way in Proposition 4.5 or Section 5.2 of \cite{azais} or in \cite{azais1}. It is stated  in term of crossings of a level and can be applied directly  to critical points  that are simply crossings of the level zero of the derivative.

 \begin{prop}  \label{p:crit}
Let $ \cX= \{X(t) : t \in  \R \} $ be a stationary and isotropic Gaussian process  with real values satisfying Assumption (A$3$).
Let  $A(\rho)$ be the correlation function  between critical points of $\cX$ defined in \eqref{e:cor}. Then, as $\rho \to 0$
   \begin{equation}    \label{e:jma:1}
   A(\rho) \simeq  \frac{( \lambda_2\lambda_6-\lambda_4^2)}{8\pi\sqrt{\lambda_4\lambda_2^3}} \rho.
   \end{equation}
Under Assumption (A$4$), let $A^{1,1}(\rho)$ (resp. $A^{0,0}(\rho)$) be the correlation function 
between maxima (resp. minima) defined in \eqref{eqfctcorr}. Then as $\rho \to 0$
 \begin{equation} \label{e:jma:2}
    A^{1,1}(\rho)=A^{0,0}(\rho) \simeq \frac{(\lambda_4 \lambda_8-\lambda_6^2)^{3/2}}{1296\pi^2(\lambda_4^2)(\lambda_2 \lambda_6-\lambda_4^2)^{1/2}} \rho^4. 
   \end{equation}
   \end{prop}
   
\noindent Note that all the coefficients above are positive.  The interpretation of the proposition  is that  we have always  repulsion between  critical points  and a very strong repulsion between maxima or between minima. As we will see, the surprising result is that does not remain true in higher dimension. 

\noindent Before beginning the proof  we need to give the following lemma .
\begin{lemme} \label{Mario}
Let $X,Y$ be two jointly Gaussian variables with common variance $\sigma^2$  and correlation $c$.  Let $r$ a positive real. Then as $c \to -1$
$$
\E \big((X^+ Y^+)^r\big) \simeq K_r \sigma^{-2(1+r)} \big(\det(\Var(X,Y))\big)^{(2r+1)/2},
$$ 
where 
 $$K_r  =   \frac{1}{2\pi} \int_0^{+\infty}  \int_0^{+\infty}  x^r y^r \exp\Big(- \frac{(x+y)^2}{2}\Big) dx dy < + \infty
 \ \ ,\ \ \displaystyle K_1=\frac{1}{6\pi}.
 $$
Moreover as $c \to -1$
$$
\E \left(X^+Y^-\right) \simeq \frac{\sigma^2}{2}.
$$
\end{lemme}
\begin{proof}[Proof of Lemma \ref{Mario}] 
We set $\Sigma:= \Var(X,Y)$. Then
\begin{align*}
\E \big((X^+ Y^+)^r\big)
 & = \int_0^{+\infty} \int_0^{+\infty} x^r y^r   \frac {1} {2\pi \sqrt{\det \Sigma} }
\exp\Big(- \frac{\sigma^2(x^2 -2c xy +y^2)}{2 \det \Sigma }\Big)  dx dy\\
&\simeq \frac {1} {2\pi} (\det \Sigma )^{(2r+1)/2}\sigma^{-2(1+r)}   \int_0^{+\infty} \int_0^{+\infty} x^r y^r  
 \exp\Big(- \frac{(x+y)^2}{2}\Big) dx dy,
\end{align*}
where we have made the change  of variables $\displaystyle x= x'\sqrt{\frac{\det \Sigma }{\sigma^2}}$, $\displaystyle y= y'\sqrt{\frac{\det \Sigma }{\sigma^2}}$
and the passage to the limit is justified because the integrand is a monotone function of $c$.  It is easy to check the convergence of the integral.
$$
\E \left(X^+ Y^-\right) = -\int_0^{+\infty} \int_{-\infty}^0   \frac {xy} {2\pi \sqrt{\det \Sigma} }
\exp\Big(- \frac{\sigma^2(x^2 -2c xy +y^2)}{2 \det \Sigma }\Big)  dy dx.
$$
Integrating first with respect to $y$ we obtain
$$
\E \left(X^+ Y^-\right) = \frac{(\det \Sigma)^{3/2}}{2\pi \sigma^4}-c\int_0^{+\infty}\frac{x^2}{\sigma \sqrt{2\pi}}\exp\left(-\frac{x^2}{2\sigma^2}\right)\Phi\left(-cx\sqrt{\frac{\sigma^2}{\det \Sigma}}\right)dx.
$$
As $c \to -1$, $\displaystyle \frac{\det \Sigma}{\sigma^2}=\sigma^2(1-c^2)\to 0$, therefore
$$
\int_0^{+\infty}\frac{x^2}{\sigma \sqrt{2\pi}}\exp\left(-\frac{x^2}{2\sigma^2}\right)\Phi\left(-cx\sqrt{\frac{\sigma^2}{\det \Sigma}}\right)dx\to \frac{\sigma^2}{2}.
$$
Moreover $\displaystyle \frac{(\det \Sigma)^{3/2}}{\sigma^6}\to 0$. That concludes the proof.
\end{proof}

\begin {proof}[Proof of Proposition \ref{p:crit}]
The first relation \eqref{e:jma:1} is a particular case of Theorem \ref{t1} below. Let us prove the second relation \eqref{e:jma:2}. By \eqref{eqfctcorr}
$$
A^{0,0}(\rho)=p_{X'(\rho),X'(0)}(0,0)\E\left(X''(0)^+ X''(\rho)^+  /  X'(0) = X'(\rho) =0\right).
$$
By \eqref{eqdetvar} we have
$$
p_{X'(\rho),X'(0)}(0,0) \simeq \rho^{-1}\left(2\pi \sqrt{\lambda_4 \lambda_2}\right)^{-1}.
$$
Define 
 \begin{eqnarray*}
w&:=&\Cov \left( X''(0),X''(\rho) /X'(0)=X'(\rho)=0 \right),\\
v&:=&\Var \left( X''(0)/X'(0)=X'(\rho)=0 \right)=\Var \left(X''(\rho) /X'(0)=X'(\rho)=0 \right),\\
D&:=&\det\left(\Var \left( X''(0),X''(\rho) /X'(0)=X'(\rho)=0  \right)\right)=v^2-w^2.
\end{eqnarray*}
By \eqref{eqcov2}, \eqref{eqvarxi1j} and \eqref{determinant1} we have $\displaystyle w \simeq -v$,
$$
v \simeq \rho^2\frac{\lambda_2 \lambda_6 -\lambda_4^2}{4\lambda_2} \mbox{ and }
D \simeq \rho^6\frac{(\lambda_4\lambda_8-\lambda_6^2)(\lambda_2 \lambda_6-\lambda_4^2)}{144 \lambda_2 \lambda_4}.
$$
Now using Lemma \ref{Mario} we get 
$$
\E\left(X''(0)^+ X''(\rho)^+  /  X'(0) = X'(\rho) =0\right) \simeq  \frac{ D^{3/2}}{6 \pi v^2}.
$$
That concludes the proof.
\end{proof}

 \subsection{Correlation function between critical points for random fields} \label{s:correlation}
\subsubsection{All the critical points}

Consider $  \cX= \{X(t) : t \in  \R^N \} $ a stationary and isotropic Gaussian field  with real values. Let us consider two points $s,t \in  \R^N $.
\noindent Theorem \ref{t1} below, gives the asymptotic expression (as $\rho \rightarrow 0$) of the correlation function between critical points \eqref{e:cor} of any isotropic Gaussian field. It generalizes
the result of \cite{BCW}  and  of \cite{BCW2} to general fields in any dimension. 
\begin{theorem}\label{t1}
Let $  \cX= \{X(t) : t \in  \R^N \} $ be a stationary and isotropic Gaussian field  with real values satisfying Assumption (A$3$). Let $A(\rho)$ be the correlation function between critical points \eqref{e:cor}.
Then as $ \rho \rightarrow 0$,
\begin{equation} \label{e:resu1}
A(\rho)\simeq \rho^{2-N}\frac{\gamma_{N-1}}{2^33^{(N-1)/2}\pi^N}\left(\sqrt{\frac{\lambda_4}{\lambda_2}}\right)^N\frac{(\lambda_2\lambda_6-\lambda_4^2)}{\lambda_2\lambda_4},
\end{equation}
where 
$\gamma_{N-1}$ is defined by 
\begin{equation*}
\gamma_{N-1}:=\E \left(
{\det}^2 \left(G_{N-1}-\Lambda Id_{N-1}\right) \right),
\end{equation*}
with $G_{N-1} $ a $(N-1)$-GOE matrix and $\Lambda$ an independent Gaussian random variable with variance $1/3$.
\end{theorem}

\begin{remark}
We set $\displaystyle \gamma_{N-1,2}(x):=\E \left({\det}^2 \left(G_{N-1}-x Id_{N-1}\right)\right)$. Note that formulas 2.2.16 and 2.2.17 in \cite{mehta1} give
explicit expressions for the calculation of $\gamma_{N-1,2}(x)$ for $x\in \R$.
\end{remark}

\begin{remark} 
\begin{itemize}
\item Of course  when  $N=1$, we retrieve  \eqref{e:jma:1}.
\item In the particular case $N=2$,  our result agrees with the result of 
\cite{BCW2}, Theorem 1.2.
\item Theorem \ref{t1} means that, for $N=2$, there is a neutrality between  critical points  and for $N>2$ there is even attraction! This is quite different from the case $N=1$.
The next  theorem will give an interpretation of this phenomenon.
\item A first  important consequence   is finiteness of the second moment  of  the number of critical points.
Indeed if $S$ is a compact set of $\R^N$ we can write the Kac-Rice formulas of order 1 and 2, \eqref {e:rice2k}, \eqref{e:riceordre2}, \eqref{eqfctcorr}. If  $\cN^c(S)$ is the number of critical points  that belong to $S$ then
$$
\E(\cN^c(S)) = |S| \E \big( |\det \nabla^2 X(t)| \big) p_{\nabla X(t)} (0) 
$$
and
\begin{equation}
\E\big ((\cN^c(S))(\cN^c(S)-1)\big) = \int_{S^2} A(||s-t||) ds  \ dt. \label{e:second}
\end{equation}
Since  $ \lambda_4$ is assumed to be  finite,   the expectation is always finite.
As for the second factorial moment,  its finiteness is in general  not systematic. In our case \eqref{e:resu1}  implies the convergence of the integral  in \eqref{e:second} on the diagonal implying in turn   the finiteness of the second moment.
Nevertheless our conditions are less general than \cite{estrade} or \cite{Al}.
\end{itemize}
\end{remark}

\begin{proof}[Proof of Theorem \ref{t1}]

For the simplicity of the exposition  we admit in a first step all passages to the limit. They will be justified at the end of the proof. 
 By \eqref{e:cor}
  \begin{equation} \label{e:jma:A}
 A(\rho) = \E \left(|\det (\xi(0))\det (\xi(t)) | \right) p_{\nabla X(0),\nabla X(t)}(0,0),
 \end{equation}
 with $t= \rho e_1$. We recall that $\xi(0) $ and $\xi(t)$ are a representation of the conditional distribution of the Hessian at $0$ and $t$. 
 Because of \eqref{eqdetvar}, 
 \begin{equation} \label{e:jma:dens}
 p_{\nabla X(0),\nabla X(t) }(0,0)  \simeq \rho^{-N} 3^{(N-1)/2} \lambda_2^{-N/2}\lambda_4^{-N/2}
 (2 \pi)^{-N}.
\end{equation}

\noindent It remains to study the expectation in \eqref{e:jma:A}.
We denote  by  $\xi_{-1}(t)$ the $(N-1)\times (N-1)$ matrix corresponding to the Hessian matrix
 $\xi(t)$ without its first row  and column.

\noindent Let us develop  $\det\left(\xi(t)\right)$ with respect to the first row, for example.
The first term is  $ \xi_{11}(t) \det\left(\xi_{-1}(t)\right)$. 

\noindent Consider  another term and develop it now with respect to the first column. Each  of the $N-1$ terms that appear   are products that include a factor  $\xi_{1j} \xi_{j'1}$ with $j,j' \neq 1$ so  because of  \eqref{eqvarxi1j} they are  $O_p(\rho^2)$: divided by $\rho^2$ they are bounded in probability. As a consequence we have proved that 
$$
\det\left(\xi(t)\right) = \xi_{11}(t) \det\left(\xi_{-1}(t)\right) + O_p(\rho^2).
$$
By \eqref{eqvar2} and \eqref{eqvarxi1j} we have
\begin{equation}\label{e:xim1}
\xi_{-1}(t)\simeq \sqrt{\frac{2\lambda_4}{3}}\left(G_{N-1}-\Lambda Id_{N-1}\right),
\end{equation} 
where $\Lambda$ is a ${\cal N}(0,1/3)$ random variable, $G_{N-1}$ is a size ($N-1$) GOE matrix defined previously, $\Lambda$ and $G_{N-1}$ are independent. So
\begin{equation}\label{eqdet2}
\det \left(\xi_{-1}(t)\right)\simeq \left(\frac{2\lambda_4}{3}\right)^{(N-1)/2}\det\left(G_{N-1}-\Lambda Id_{N-1}\right).
\end{equation}
The order of magnitude of  the first term in the development of the determinant  is then  $\rho$. 

\noindent In  conclusion we have proved that 
\begin{align} \label{decadix}
\det\left(\xi(t)\right) & = O_p(\rho) =  \xi_{11}(t) \det\left(\xi_{-1}(t)\right)   +  O_p(\rho^2)  \\
\det\left(\xi(0)\right) \det \left(\xi(t)\right) & =  O_p(\rho^2)   = \xi_{11}(0)\xi_{11}(t) \det\left(\xi_{-1}(0)\right)\det\left(\xi_{-1}(t)\right) + O_p(\rho^3) \notag.
\end{align}

\noindent By \eqref{eqcov2}, \eqref{eqvarxijk} and \eqref{zero1}:
\begin{eqnarray*}
&\mbox{Corr}(\xi_{11}(0),\xi_{11}(t))\to -1, \ \mbox{Corr}(\xi_{jk}(t),\xi_{jk}(0))\to 1,\  \forall j,k\in {2,\ldots,N},&\\
&\Cov(\xi_{11}(t),\xi_{jj}(t)) = O(\rho^2),\ \Cov(\xi_{11}(t),\xi_{jk}(t))= 0 \ \forall j\neq k\in {1,\ldots,N}.&
\end{eqnarray*}
Then
\begin{equation}
\det\left(\xi(0)\right) \det \left(\xi(t)\right) \simeq -\xi_{11}(t)^2 \mbox{det}^2\left(\xi_{-1}(t)\right)  \label{e:det:1}
\end{equation}
$$
\E \left(|\det\left(\xi(0)\right) \det \left(\xi(t)\right)|\right) \simeq \Var(\xi_{11}(t)) 
\E \left( \mbox{det}^2\left(\xi_{-1}(t)\right)\right),
$$
Finally equations \eqref{eqvarxi1j} and \eqref{eqdet2} with \eqref{e:jma:dens} and \eqref{e:jma:A} give \eqref{e:resu1}.\bigskip

\noindent Justification of the passages to the limit: Since we are in the Gaussian space generated by the random field $\cX$ all the variables considered above are jointly Gaussian.
So their absolute values   are bounded   by the maximum in absolute value of a bounded random field.
General  results about the maximum of random fields, for example the Borell-Sudakov-Tsirelson  theorem (\cite{azais} Section 2.4)  implies that the  maximum  of its absolute value  has moments  of every order, giving all the dominations needed. 
\end{proof}
%
\subsubsection{Correlation function between critical points with adjacent indexes}  
\begin{theorem}\label{t2}  
Let $  \cX= \{X(t) : t \in  \R^N \} $ be a stationary and isotropic Gaussian field  with real values satisfying Assumption (A$3$). Let 
$A^{k,k+1}(\rho)$ be the correlation function \eqref{eqfctcorr} between critical points with index $k$ and critical points with index $k+1$, $k=0,\ldots,N-1$. Then as $\rho \rightarrow 0$,
\begin{equation}
A^{k,k+1}(\rho)\simeq \rho^{2-N}\frac{\gamma_{N-1}^k}{2^43^{(N-1)/2}\pi^N}\left(\sqrt{\frac{\lambda_4}{\lambda_2}}\right)^N\frac{(\lambda_2\lambda_6-\lambda_4^2)}{\lambda_2\lambda_4}, \label{resu2}
\end{equation}
where $\gamma_{N-1}^k$ is defined by 
 \begin{equation}\label{e:gammaN_1_k}
\gamma_{N-1}^k:=\E \left(
{\det}^2 \left(G_{N-1}-\Lambda Id_{N-1}\right) \1_{\{i(G_{N-1}-\Lambda Id_{N-1})=k\}}\right),
\end{equation}
with $G_{N-1} $ a $(N-1)$-GOE matrix and $\Lambda$ an independent centred Gaussian random variable with variance $1/3$.
 \end{theorem}
 
This theorem  gives an interpretation  to Theorem  \ref{t1}: the attraction for $N \geq 3$  is in fact due to  attraction between  critical points with adjacent indexes.\\

Again, in the particular case of the random plane wave ($N=2$),  our result agrees with the result of 
\cite{BCW}, formula (9).\\

We set $\displaystyle \gamma_{N-1,2}^k(x):=\E\left({\det}^2 \left(G_{N-1}-x Id_{N-1}\right) \1_{\{i(G_{N-1}-x Id_{N-1})=k\}}\right)$ for $x\in \R$.
Before proving Theorem \ref{t2} we give in the following lemma an exact expression for the calculation of $\displaystyle \gamma_{N-1,2}^k(x)$. 

\begin{lemme}\label{lemmegammaN_1_k}
Using notation above and notation introduced in paragraph \ref{paragrapheGOE} we have
$$
\gamma_{N-1,2}^k(x)=k_{N-1}(N-1)!\sum_{I\in {\cF}_k(N-1)}{  \rm Pf }\left({\cA}^{2}(I,\ell)\right),
$$
where $k_{N-1}$ is given by \eqref{eqkn} and ${  \rm Pf }\left({\cA}^{2}(I,\ell)\right)$ is the Pfaffian of the skew matrix ${\cA}^{2}(I,\ell)$ defined by \eqref{def:Atilde1} and \eqref{def:Atilde2} with $n=N-1$ and $\alpha=2$.
\end{lemme}

\begin{proof}[Proof of Lemma \ref{lemmegammaN_1_k}]
The proof follows the same lines as in Theorem \ref{lemmedensitevpk} with 
\begin{multline*}
\gamma_{N-1,2}^k(\ell)=k_{N-1}(N-1)!\int_{{\cal O}_{k+1}(\ell)} \prod_{i=1}^{N-1}(\mu_i-\ell)^2\\
\prod_{1\leq i<j\leq N-1}(\mu_j-\mu_i)\exp\left(-\frac{\sum_{i=1}^{N-1}\mu_i^2}{2}\right)d\mu_1\ldots d\mu_{N-1}.
\end{multline*}
Then we set ${\bf{H}}({\pmb{\mu}},\ell)$ the matrix $\{h_m(\mu_j,\ell),m,j=1,\ldots,N-1\}$ with
$h_{m}(\mu_j,\ell)=\mu_j^{m-1}(\mu_j-\ell)^2\exp\left(-\mu_j^2/2\right)$.
\end{proof}

\begin{proof}[Proof of Theorem \ref{t2}]
By \eqref{eqfctcorr} we have
\begin{eqnarray*}
A^{k,k+1}(\rho)&=&p_{\nabla X(0),\nabla X(t)}(0,0)\E\left(|\det \xi(0)|\1_{\{ i(\xi(0)) = k\}}|\det \xi(t)|\1_{\{i(\xi(t)) = k+1\}}\right) \\
&=&-p_{\nabla X(0),\nabla X(t)}(0,0)\E\left(\det \left(\xi(0)\right)\1_{\{ i(\xi(0)) = k\}}\det \xi(t)\1_{\{i(\xi(t)) = k+1\}}\right).
\end{eqnarray*}
We can modify the computation of the determinant of $\xi(0)$  in \eqref{decadix}  to compute in place $ \det\Big(\xi(0) - \lambda Id_N\Big)$:
\begin{equation}\label{e:nabla}
\det\left(\xi(0) - \lambda Id_N\right) = 
\big(\xi_{11}(0) -\lambda\big) \det\left(\xi_{-1}(0) -\lambda Id_{N-1}\right) + O_p(\rho^2).
\end{equation}
Note  that as $\rho  \to 0$, $\xi (0)$ converges  in $L^2$  (or a.s. if we use a Skorohod imbedding argument)
 to 
 $$
 \left(\begin{array}{cccc}
 0 & \quad 0 & \ldots  & 0 \\
 0  \\
 \vdots  & & \xi_{-1}(0) \\
 0
  \end{array}\right).
  $$
Its eigenvectors converge  to those  of this  last matrix. These eigenvectors are associated to different  eigenvalues so we can define properly  the eigenvalue $\mu_1$  that tends to zero.
Because of  \eqref{e:nabla} with $\lambda = \mu_1$:
$$
(\xi_{11} (0) -\mu_1) \det\left(\xi_{-1}(0) -\mu_1 Id_{N-1}\right)  = O_p(\rho^2),
$$
implying in turn that 
$$
\mu_1 = \xi_{11} (0) +  O_p(\rho^2).
$$
On the other hand, the others eigenvalues, say $\mu_2,\ldots,\mu_{N}$, converge in  distribution to that 
 given by the right-hand side  of \eqref{e:xim1}. This implies in turn that 
 $$
 \1_{\{ i(\xi(0)) = k\}} \simeq 
   \1_{\{ \xi_{11}(0) <0; i(\xi_{-1}(0)) = k-1 \}} +\1_{\{\xi_{11}(0) >0; i(\xi_{-1}(0)) = k\}}.
$$
 
\noindent As a consequence, when computing $\E\left(\det \left(\xi(0)\right)\1_{\{ i(\xi(0)) = k\}}\det \xi(t)\1_{\{i(\xi(t)) = k+1\}}\right)$,  we have four cases to consider  depending on the signs of  $ \xi_{11} (0), 
 \xi_{11} (t)$.
 First  we  consider  the case $ \xi_{11} (0) >0,  \xi_{11} (t) <0$. 
 We  have  the equivalent of \eqref{e:det:1} 
\begin{multline*}
 \det\left(\xi(0)\right) \det \left(\xi(t)\right)  \1_{\{\xi_{11} (0) >0,  \xi_{11} (t)<0 \}} \1_{\{i(\xi_{-1}(t))=i(\xi_{-1}(0))=k \}}\\
 =  -\xi_{11}(t)^2 \1_{\{\xi_{11}(t)<0\}} \mbox{det}^2\left(\xi_{-1}(t)\right)\1_{\{i(\xi_{-1}(t))=k \}}  + O_p(\rho^3),
\end{multline*}
 giving 
\begin{multline*}
  \E \left(|\det\left(\xi(0)\right) \det \left(\xi(t)\right)| \1_{\{\xi_{11} (0) >0,  \xi_{11} (t)<0\}}\1_{\{i(\xi_{-1}(t))=i(\xi_{-1}(0))=k \}}  \right) \\
  =  1/2\Var(\xi_{11}(t)) \E \left( \mbox{det}^2\left(\xi_{-1}(t)\right)\1_{\{i(\xi_{-1}(t))=k \}}\right)  + O(\rho^3).
\end{multline*}
This gives  \eqref{resu2} as soon  as we have checked that the three other cases  give smaller contributions  which is direct. 
 
\end{proof}

\subsubsection{Correlation function between critical points with extreme indexes}

First, we give a bound for the correlation function between maxima and minima (Theorem \ref{theo3})
then a bound for the correlation function between maxima (or between minima) (Theorem \ref{theo4}).

\begin{theorem}\label{theo3} 
Let $  \cX= \{X(t) : t \in  \R^N, N>1\}$ be a stationary and isotropic Gaussian field  with real values satisfying Assumption (A$3$). Let 
$A^{0,N}(\rho)$ be the correlation function \eqref{eqfctcorr} between minima and maxima. Then
for all $\epsilon >0$ there exists a  positive constant $K(\epsilon)>0$ such that as $\rho \rightarrow 0$,
\begin{equation}
  A^{0,N}(\rho)\leq K(\epsilon) \times \rho^{2N-1-\epsilon}. \label{resu3}
\end{equation}
\end{theorem}

Upper bound \eqref{resu3} proves strong repulsion between maxima and minima.

\begin{theorem}\label{theo4}
Let $  \cX= \{X(t) : t \in  \R^N, N>1 \} $ be a stationary and isotropic Gaussian field  with real values satisfying Assumption (A$4$). Let 
$A^{N,N}(\rho)$ (respectively $A^{0,0}(\rho)$) be
the correlation function  \eqref{eqfctcorr} between maxima (respectively minima). Then
$\forall  \epsilon >0$ there exists a constant $\tilde{K}(\epsilon )>0$ such that as $\rho \rightarrow 0$,
\begin{equation}
 A^{0,0}(\rho)=A^{N,N}(\rho)\leq \tilde{K}(\epsilon )\times \rho^{5-N-\epsilon }. \label{resu4}
\end{equation}
\end{theorem}

Upper bound \eqref{resu4} proves repulsion between maxima (or minima) for $N<5$. 

In the particular case $N=1$, note that, strictly speaking, Proposition \ref{p:crit} gives  no equivalent of Theorem  \ref{theo3}. In fact, since  a maxima is followed by a minima and reciprocally it is easy  to deduce from  \eqref{e:jma:1} that we have \eqref{resu3} with $\epsilon =0$. On the other hand \eqref{e:jma:2}  implies that the result of Theorem \ref{theo4} holds true  with $\epsilon =0$.

As this work was achieved we have discovered the  work \cite{BCW2} which is similar to the result of this section but limited to the dimension $N=2$.  In this last   particular case Theorem \ref{t1} is the analogue of Theorem 1.2 of  \cite{BCW2};
Theorem \ref{t2}  has no equivalent  in   \cite{BCW2}; Theorems \ref{theo3} and \ref{theo4} are similar but a little weaker  than 
Theorems 1.3 and 1.4 of  \cite{BCW2}  that are obtained  by a difficult  diagonalization of the Hessian  which is absolutely impossible in the case $N>2$.

\begin{proof}[Proof of Theorem \ref{theo3}]

\noindent We use Hadamard's inequality. If $M $  is a positive semidefinite matrix of size $N$
$$
\det(M) \leq M_{11} \times \cdots \times M_{NN}.$$
As a consequence  for any symmetric matrix $M'$
\begin{equation}
\det(M') \1_{\{i(M')=0\}}  \leq   (M'_{11})^+ \times \cdots \times( M'_{NN})^+ \ . \label{eqdetM}
\end{equation}
By Kac-Rice formula \eqref{eqfctcorr}
\begin{equation*}
A^{0,N}(\rho)=p_{\nabla X(0),\nabla X(t)}(0,0)\E\left(|\det \xi(0)|\1_{\{ i(\nabla^2 \xi(0)) = 0\}}|\det \xi(t)|\1_{\{i(\nabla^2 \xi(t)) = N\}}\right).
\end{equation*}
We have
$$
A^{0,N}(\rho) \leq (-1)^N p_{\nabla X(0),\nabla X(t)}(0,0)\E\left(\xi_{11}^+(0)\xi_{22}^+(0)\ldots \xi_{NN}^+(0)\xi_{11}^-(t)\xi_{22}^-(t)\ldots \xi_{NN}^-(t)\right).
$$
By Cauchy-Schwarz inequality and symmetry of the role of $0$ and $t$ we obtain
$$
A^{0,N}(\rho) \leq  (-1)^{N-1}p_{\nabla X(0),\nabla X(t)}(0,0)\E\left(\xi_{11}^+(0)^2\xi_{22}^+(0)\ldots \xi_{NN}^+(0)\xi_{22}^-(t)\ldots \xi_{NN}^-(t)\right).
$$
By H\"older's inequality for every  $p>1$
$$
A^{0,N}(\rho) \leq  p_{\nabla X(0),\nabla X(t)}(0,0)\times I_{\rho} \times \left(\E\left(\xi_{11}^+(0)^{2p/(p-1)}\right)\right)^{(p-1)/p}.
$$
where $\displaystyle I_{\rho}:=\left(\E\left(|\xi_{22}^+(0)\ldots \xi_{NN}^+(0)\xi_{22}^-(t)\ldots \xi_{NN}^-(t)|^p\right)\right)^{1/p}$.
Since $\xi_{11}(0)$ is centred Gaussian with variance $\displaystyle \rho^2 \frac{\lambda_2\lambda_6-\lambda_4^2}{4\lambda_2}$, there exists a constant $K(p)>0$ depending on $p$ such that
$$
\left(\E\left(\xi_{11}^+(0)^{2p/(p-1)}\right)\right)^{(p-1)/p}\leq \rho^2K(p).
$$
By \eqref{e:jma:dens} we have $\displaystyle p_{\nabla X(0),\nabla X(t)}(0,0)=O(\rho^{-N})$. Now let us prove that
\begin{equation}
I_{\rho} =O\left(\rho^{(N-1)(2+1/p)}\right).\label{eqI}
\end{equation}
We set  $ \tilde \xi:=\left(\xi_{22}(0),\ldots,\xi_{NN}(0),\xi_{22}(t),\ldots,\xi_{NN}(t)\right)$, $\tilde \xi_i$ the $i$th coordinate of $\tilde \xi$ for $i=1,\ldots,2N-2$,
$\tilde \xi_{-i}$ (resp. $\tilde \xi_{-(i,j)}$) the vector $\tilde \xi$ without its $i$th (resp. $i$th and $j$th) coordinate(s).
We set $\Sigma_\rho:=\Var (\tilde \xi)$. We define ${\cal S}=\{x\in \R^{2N-2}:x_1>0,\ldots,x_{N-1}>0,x_N<0,\ldots,x_{2N-2}<0\}$. Then
\begin{equation}\label{Irho1}
I_{\rho}=\left(\int_{{\cal S}
} \frac{|\tilde \xi_{1}\ldots \tilde \xi_{2N-2}|^p}{(2\pi)^{(2N-2)/2}\sqrt{\det \Sigma_{\rho}}}\exp\left(-\frac{1}{2}\tilde \xi \Sigma_{\rho}^{-1} \tilde \xi^\top  \right)d\tilde \xi\right)^{1/p}.
\end{equation}
\noindent $\bullet$ {\underline {Equivalent of $\det (\Sigma_\rho)$ as $\rho \rightarrow 0$}}\\
\\
\noindent For short we set $X_i:=X_i(0)$ and $X_{ij}:=X_{ij}(0)$.
$$
\det \Sigma_\rho=\frac{\det \Var(X_{22},\ldots,X_{NN},X_{22}(t),\ldots,X_{NN}(t),X_1,\ldots,X_{N},X_1(t),\ldots,X_{N}(t))}{\det \Var(X_1,\ldots,X_{N},X_1(t),\ldots,X_{N}(t))}.
$$
Using Trick \ref{trick1}, we obtain as $\rho \rightarrow 0$
\begin{equation*}
\det \Sigma_\rho \simeq \rho^{2(N-1)}\frac{\det \Var(X_{22},\ldots,X_{NN},X_{122},\ldots,X_{1NN},X_1,\ldots,X_N,X_{11},\ldots,X_{1N})}{\det \Var (X_1,\ldots,X_N, X_{11},\ldots, X_{1N})}.
\end{equation*}
Using Lemma \ref{lemme1}, \eqref{equation:determinant1} and \eqref{equation:determinant3} we get, as $\rho \rightarrow 0$
\begin{align*}
\det \Sigma_\rho \simeq& \rho^{2(N-1)} \frac{\det \Var(X_{11},\ldots,X_{NN}) \det \Var(X_1,X_{122},\ldots,X_{1NN})}{ \Var (X_1)\Var(X_{11})}\\
\simeq& \rho^{2(N-1)}(N+2)\left(\frac{2\lambda_4}{3}\right)^{N-1}\left(\frac{2\lambda_6}{15}\right)^{N-2}\left(\frac{3(N+1)\lambda_2\lambda_6-5(N-1)\lambda_4^2}{135\lambda_2}\right).
\end{align*}
So
$$
\rho^{2(1-N)} \det \Sigma_\rho \simeq g:=(N+2)\left(\frac{2\lambda_4}{3}\right)^{N-1}\left(\frac{2\lambda_6}{15}\right)^{N-2}\left(\frac{3(N+1)\lambda_2\lambda_6-5(N-1)\lambda_4^2}{135\lambda_2}\right).
$$
By \eqref{eq0:lemme1} we have $\displaystyle \lambda_2\lambda_6\geq \frac{5(2+N)}{3(4+N)}\lambda_4^2$. Since $\displaystyle \frac{5(N-1)}{3(N+1)}<\frac{5(2+N)}{3(4+N)}$ we deduce that $g>0$ under our hypotheses.\\
\\
\noindent $\bullet$ {\underline {Equivalent of $\Sigma_\rho^{-1}$ as $\rho \rightarrow 0$}}.\\
\\
Because of the exchangeability  of the coordinates $2,\ldots,N$
$$
\Sigma_\rho^{-1}:=\left(
\begin{array}{c|c}
 \alpha_{\rho}Id_{N-1}+\beta_{\rho}J_{N-1,N-1} & \gamma_{\rho}Id_{N-1}+\delta_{\rho}J_{N-1,N-1} \\
 \hline
 \gamma_{\rho}Id_{N-1}+\delta_{\rho}J_{N-1,N-1} & \alpha_{\rho}Id_{N-1}+\beta_{\rho}J_{N-1,N-1}
\end{array}
\right).
$$
\noindent Using the classical expression  of conditional variance and covariance
\begin{align*}
\left(\Sigma_\rho^{-1}\right)_{ii}&=\frac{1}{\Var(\tilde \xi_{i}/\tilde \xi_{-i})}=\frac{\det \Var(\tilde \xi_{-i})}{\det \Var (\tilde \xi)}\\
\left(\Sigma_\rho^{-1}\right)_{ij}&=\frac{-\Cov(\tilde \xi_i,\tilde \xi_j / \tilde \xi_{-(i,j)})}{\det \Var(\tilde \xi_i,\tilde \xi_j/ \tilde \xi_{-(i,j)})}
=\frac{\Var(\tilde \xi_{i}/ \tilde \xi_{-(i,j)})+ \Var (\tilde \xi_{j}/ \tilde \xi_{-(i,j)})- \Var(\tilde \xi_i+\tilde \xi_j/ \tilde \xi_{-(i,j)})}{2\det \Var(\tilde \xi_i,\tilde \xi_j/ \tilde \xi_{-(i,j)})}\\
&=\frac{\det \Var(\tilde \xi_{-j})+\det \Var (\tilde \xi_{-i})-\det \Var(\tilde \xi_i+\tilde \xi_j,\tilde \xi_{-(i,j)})}{2\det \Var(\tilde \xi)},
\end{align*}
and using the same techniques as those previously used for the calculation of the equivalent of $\det (\Sigma_\rho)$ and for the proof of Lemma \ref{lemme2},
we obtain the equivalents of $\left(\Sigma_\rho^{-1}\right)_{11}$, $\left(\Sigma_\rho^{-1}\right)_{12}$, $\left(\Sigma_\rho^{-1}\right)_{1N}$ and $\left(\Sigma_\rho^{-1}\right)_{1(N+1)}$.  \\
We get
$\displaystyle \rho^2 \alpha_{\rho} \simeq \alpha := \frac{15}{2\lambda_6}$,
$\displaystyle \rho^2 \beta_{\rho} \simeq \beta := \frac{15}{2\lambda_6} \frac{5\lambda_4^2-3\lambda_2\lambda_6}{3(N+1)\lambda_2\lambda_6-5(N-1)\lambda_4^2}$,
$\displaystyle \rho^2 \gamma_{\rho} \simeq -\alpha$ and $\displaystyle \rho^2 \delta_{\rho} \simeq -\beta$. So $\rho^2 \Sigma_{\rho}^{-1}$ converges to
$$ \tilde{\Sigma}:=\left(
\begin{array}{c|c}
 \alpha Id_{N-1}+\beta J_{N-1,N-1} & -\alpha Id_{N-1}-\beta J_{N-1,N-1} \\
 \hline
 -\alpha Id_{N-1}-\beta J_{N-1,N-1} & \alpha Id_{N-1}+\beta J_{N-1,N-1}
\end{array}
\right)
$$
where the matrix $\displaystyle \alpha Id_{N-1}+\beta J_{N-1,N-1}$ is positive definite under our conditions. Indeed by \eqref{eqdetxid+yj}
$$
\det(\alpha Id_{n}+\beta J_{n,n})=\alpha^{n-1}(\alpha+n\beta),
$$
and we have $\alpha>0$ and $$\alpha+n\beta=\frac{15}{2\lambda_6}\frac{3(N+1-n)\lambda_2\lambda_6-5(N-1-n)\lambda_4^2}{3(N+1)\lambda_2\lambda_6-5(N-1)\lambda_4^2}>0\mbox{ }\forall n=1,\ldots, N-1$$
under our conditions.
\\
\\
\noindent $\bullet$ {\underline {Formal finite limit}}\\
\\
We make the change of variables $\tilde \xi =\rho z$ in \eqref{Irho1} to obtain
\begin{multline}
I_{\rho}=\rho^{(2N-2)+(N-1)/p} \\
\times \left(\int_{{\cal S}
} \frac{|z_{1}\ldots z_{2N-2}|^p}{(2\pi)^{(2N-2)/2}\sqrt{\rho^{2(1-N)}\det \Sigma_{\rho}}}\exp\left(-\frac{1}{2}z \left(\rho^2 \Sigma_{\rho}^{-1}\right) z^\top  \right)dz\right)^{1/p}.\label{integrand}
\end{multline}
We set $\bar{z}_1:=(z_1,\ldots,z_{N-1})$, $\bar{z}_2:=(z_N,\ldots,z_{2N-2})$ and
$$
f_{\rho}(z):=\frac{|z_{1}\ldots z_{2N-2}|^p}{(2\pi)^{(2N-2)/2}\sqrt{\rho^{2(1-N)}\det \Sigma_{\rho}}}\exp\left(-\frac{1}{2}z \left(\rho^2 \Sigma_{\rho}^{-1}\right) z^\top  \right).
$$
As $\rho \rightarrow 0$, $f_{\rho}(z)$ converges to 
$
\displaystyle \frac{|z_{1}\ldots z_{2N-2}|^p}{(2\pi)^{(2N-2)/2}\sqrt{g}}\exp\left(-\frac{1}{2}z \tilde{\Sigma} z^\top  \right)
$
equals to
$$
\frac{|z_{1}\ldots z_{2N-2}|^p}{(2\pi)^{(2N-2)/2}\sqrt{g}}\exp \left(-\frac{1}{2}(\bar{z}_1-\bar{z}_2)( \alpha Id_{N-1}+\beta J_{N-1,N-1}) (\bar{z}_1-\bar{z}_2)^\top  \right)
$$
that is integrable on ${\cal S}$ since $g>0$ and $\displaystyle \alpha Id_{N-1}+\beta J_{N-1,N-1}$ is positive definite.\\
\\
\noindent $\bullet$ {\underline {Domination}}\\
\\
Let us choose $g_0$, $\alpha_0$ and $\beta_0$ such that $0<g_0<g$, $0<\alpha_0<\alpha$ and $\beta_0<\beta$ such that $0<\alpha_0+n\beta_0$ $\forall n=1,\ldots, N-1$
which is always possible. Then for $\rho$ sufficiently small
\begin{equation}\label{majoration}
f_{\rho}(z)\leq \frac{|z_{1}\ldots z_{2N-2}|^p}{(2\pi)^{(2N-2)/2}\sqrt{g_0}}\exp \left(-\frac{1}{2}(\bar{z}_1-\bar{z}_2)( \alpha_0 Id_{N-1}+\beta_0 J_{N-1,N-1}) (\bar{z}_1-\bar{z}_2)^\top  \right).
\end{equation}
Since $\alpha_0>0$ and $\alpha_0+n\beta_0>0$ $\forall n=1,\ldots, N-1$, $\displaystyle \alpha_0 Id_{N-1}+\beta_0 J_{N-1,N-1}$ is positive definite and the right hand side of \eqref{majoration}
is integrable on ${\cal S}$. Finally we conclude the proof taking $p=1-\epsilon$.
\end{proof}

\begin{proof}[Proof of Theorem \ref{theo4}]

\vspace{0.5cm}

\noindent We follow the same lines as  in the proof of Theorem \ref{theo3}.
By Kac-Rice formula \eqref{eqfctcorr}:
\begin{equation*}
A^{0,0}(\rho)=p_{\nabla X(0),\nabla X(t)}(0,0)\E\left(|\det \xi(0)|\1_{\{ i(\nabla^2 \xi(0)) = 0\}}|\det \xi(t)|\1_{\{i(\nabla^2 \xi(t)) = 0\}}\right).
\end{equation*}
Using \eqref{eqdetM}, Cauchy-Schwarz inequality and symmetry of the role of $0$ and $t$ we get
\begin{equation*}
A^{0,0}(\rho) \leq p_{\nabla X(0),\nabla X(t)}(0,0)\E\left(\xi_{11}^+(0)\xi_{11}^+(t)\xi_{22}^+(0)^2\ldots \xi_{NN}^+(0)^2\right).
\end{equation*}
By \eqref{e:jma:dens} we have $\displaystyle p_{\nabla X(0),\nabla X(t)}(0,0)=O(\rho^{-N})$. Now let us prove that
\begin{equation}
\breve{I}_{\rho}:=\E\left(\xi_{11}^+(0)\xi_{11}^+(t)\xi_{22}^+(0)^2\ldots \xi_{NN}^+(0)^2\right)=O(\rho^{4+1/p}).\label{e:i2}
\end{equation}
By H\"older's inequality for every  $p>1$
\begin{equation} \label{eq:Iteche}
\breve{I}_{\rho}\leq \left(\E\left(\xi_{11}^+(0)^p\xi_{11}^+(t)^p\right)\right)^{1/p}\left(\E\left(\xi_{22}^+(0)^{2p/(p-1)}\ldots \xi_{NN}^+(0)^{2p/(p-1)}\right)\right)^{(p-1)/p}.
\end{equation}
Since $\Cov (\xi_{11}(0),\xi_{11}(t))\simeq -\Var(\xi_{11}(t))$ (see Lemma \ref{lemme2}), by Lemma \ref{Mario}
\begin{equation}\label{eq:x_110x_11t}
\E\left(\xi_{11}^+(0)^p\xi_{11}^+(t)^p\right)\simeq \mbox{(const)}\frac{1}{(\Var(\xi_{11}(t)))^{1+p}}(\det(\Var(\xi_{11}(0),\xi_{11}(t))))^{p+1/2}.
\end{equation}
By Lemma \ref{lemme2} we have
\begin{align*}
\Var(\xi_{11}(t))&\simeq \frac{\rho^2}{4}\frac{(\lambda_2\lambda_6-\lambda_4^2)}{\lambda_2},\\
\det(\Var(\xi_{11}(0),\xi_{11}(t)))&\simeq \frac{\rho^6}{144} \frac{(\lambda_2\lambda_6-\lambda_4^2)(\lambda_4\lambda_8-\lambda_6^2)}{\lambda_2\lambda_4},
\end{align*}
and by \eqref{eq:x_110x_11t} we deduce
\begin{equation}\label{conclusion1}
\left(\E\left(\xi_{11}^+(0)^p\xi_{11}^+(t)^p\right)\right)^{1/p}\simeq \mbox{(const)}\rho^{4+1/p}.
\end{equation}
We now consider the term $\displaystyle \left(\E\left(\xi_{22}^+(0)^{2p/(p-1)}\ldots \xi_{NN}^+(0)^{2p/(p-1)}\right)\right)^{(p-1)/p}$ in \eqref{eq:Iteche}.\\
We set  $\displaystyle \breve{\xi}:=\left(\xi_{22}(0),\xi_{33}(0),\ldots,\xi_{NN}(0)\right)$.
$\breve{\xi}$ is a Gaussian vector centred with variance-covariance matrix 
$$
\Var(\breve{\xi})=(a_{\rho}Id_{N-1}+b_{\rho}J_{N-1,N-1})
$$
explicitly given in Lemma \ref{lemmeconditionaldistribution}. According to Lemma \ref{lemme2}, as $\rho \rightarrow 0$,
$$
a_{\rho} \rightarrow  \frac{2\lambda_4}{3} \mbox{ and }b_{\rho} \rightarrow \frac{2\lambda_4}{9}.
$$
Using Lebesgue's dominated convergence theorem we can deduce that there exists a constant $C(p)>0$ such that, as $\rho \rightarrow 0$
\begin{equation}\label{conclusion3}
\E\left(\xi_{22}^+(0)^{2p/(p-1)}\ldots \xi_{NN}^+(0)^{2p/(p-1)}\right) \simeq C(p).
\end{equation}
\eqref{conclusion3} and \eqref{conclusion1} give \eqref{e:i2}. Finally we conclude the proof taking $p=1-\epsilon$.
\end{proof}

\appendix
\section{Validity of Kac-Rice formulas for stationary Gaussian processes and fields} \label{a:1}

  
 
 %
Let ${\cX}=\{X(t):t\in U\subset \R^N\}$ be a zero-mean, stationary Gaussian random field defined on an open set $U\subset \R^N$. We follow chapter 6 of \cite {azais}. So we need $\nabla X(\cdot )$  to be $C^1$.  
  Recall that $\mathcal N^c(S)$, $\mathcal N_k^c(S)$ and $\mathcal N_k^c(u,S)$ are respectively the number of critical points, the number of critical points with index $k$
and the number of critical points with index $k$ above the level $u$ of the random field  $X(\cdot)$  in a Borel set $S\subset U$ .
 
 \begin{prop} \label{p:riceN}
 Let ${\cX}=\{X(t):t\in U\subset \R^N\}$ be a zero-mean, stationary Gaussian random field defined on an open set $U\subset \R^N$.
Suppose that $\cX$  is $C^2$ and that the distribution of $\nabla^2 X(t)$ does not degenerate,  then the Kac-Rice formula of order  one is always true  in the sense that for every Borel set $S \subset U$
 \begin{align}
 \E\left({\cN}^c(S)\right)&=|S| p_{\nabla X(t)}(0)  \E \left( | \det(\nabla^2 X(t))| \right)  \label{eq:rice1}   \\
 \E\left({\cN}_k^c(S)\right)&=|S|  p_{\nabla X(t)}(0)  \E \left( | \det(\nabla^2 X(t))|  \1_{i(\nabla^2 X(t)) =k} \right)  \label{eq:rice2} 
\end{align}
\begin{multline}
 \E\left({\cN}_k^c(u,S)\right)=|S|  p_{\nabla X(t)}(0) \\
\times \int_u^{+\infty}p_{X(t)}(x) \E \left( | \det(\nabla^2 X(t))|  \1_{i(\nabla^2 X(t)) =k} /X(t)=x \right) . \label{eq:rice3}
 \end{multline}
  In addition the sample paths are  a.s.  Morse functions. 
\end{prop}
Note that Lemma \ref{lemme1} implies that  the conditions  of Proposition  \ref{p:riceN} are met for a random field  satisfying (A2).
\begin{remark}
In the particular case $N=1$, the Kac-Rice formulas \eqref{eq:rice1}, \eqref{eq:rice2} and \eqref{eq:rice3} are always true (both sides are equal, finite or infinite) when $\cX$ is a zero-mean stationary Gaussian process
with $\cC ^1$ covariance function. This can be proved as in \cite{CL}. These conditions are weaker than those in Proposition \ref{p:riceN}.
\end{remark}

 \begin{proof}[Proof of Proposition \ref{p:riceN}]
We first consider the mean number of critical points.
We need to apply Theorem 6.2 of \cite {azais}. The conditions are clearly  verified in our case except for the condition (iv) :
 $$
 \P\{\exists   t \in S : \nabla X(t) =0, \det(\nabla^2 X (t)) =0\} =0.
$$
This  last condition is equivalent to the fact that the sample paths are a.s. Morse functions.
In  the case of stationary random fields  the simplest  is to use  Proposition 6.5 of \cite{azais} with condition b). Note that  because of stationarity,  by an extension argument, we can get rid of the compactness  condition. Since $\nabla X(t)$ and $\nabla^2 X(t)$ are independent, this condition is equivalent to 
$$
 \P\{|\det(\nabla^2 X(t))| <\delta \} \to 0 \mbox{ as } \delta \to 0.
 $$ which is equivalent  to 
 $$
  \P\{|\det(\nabla^2 X(t))| = 0\} =0.
$$
 This is implied by assuming the non degeneracy of $\nabla^2 X(t)$. \medskip
 
We consider now the case  of  Kac-Rice formula  for the  mean number of critical points with a given index and above  the level $u$ .  We have to apply Theorem 6.4 of \cite{azais}.  More precisely we have to write their formula (6.6)  with 
 $$
 g(t, Y^t) = g(\nabla^2 X(t) ,X(t)) ,
 $$
here the process $Y^t$  is simply the Gaussian vector $\nabla^2 X(t) ,X(t)$ which is Gaussian as required in Theorem 6.4 of \cite{azais}.
On the other hand,  $g$ is  the indicator of the index multiplied by the indicator  function that $X(t) >u$. It  is not  continuous  and we must apply   an approximation with a continuous function, followed by a monotone convergence argument.
Then, since we have independence between $X(t)$ and $\nabla X(t)$ and between $\nabla X(t)$ and $\nabla^2X(t)$ (see Lemma \ref{lemme1}), we obtain \eqref{eq:rice2} and \eqref{eq:rice3}.
\end{proof}

 For simplification we limit   our second proposition  to  the isotropic case 
 \begin{prop}\label{p:riceN2}
  Under Assumption  (A2),   the Kac-Rice formula of order 2  is valid  for a sufficiently small set $S$  in the sense that  if $S_1$ and $S_2$ are sufficiently small and sufficiently close  to each other,  
  \begin{multline}\label{eq:rice4}
 \E\left({\cN}^c(S_1){\cN}^c(S_2)\right)= \int _{S_1\times S_2}dt_1dt_2  p_{\nabla X(t_1) \nabla X(t_2) } (0,0)  \\  
\E \left( | \det(\nabla^2 X(t_1)) \det(\nabla^2 X(t_2))| / \nabla X(t_1) =\nabla X(t_2) =0\right)  
\end{multline}
\begin{multline}\label{eq:rice5}
  \E\left({\cN}_k^c(S_1){\cN}_{k'}^c(S_2)\right) 
  = \int _{S_1\times S_2}dt_1 dt_2 p_{\nabla X(t_1) \nabla X(t_2)    } (0,0)  \\  
   \E \left( | \det(\nabla^2 X(t_1)) \1_{i(\nabla^2 X(t_1)) =k} \det(\nabla^2 X(t_2))\1_{i(\nabla^2 X(t_2)) =k'}| / \nabla X(t_1) =\nabla X(t_2) =0\right)  
  \end{multline}
\end{prop}
 \begin{proof}
 We need to apply Theorem 6.3 from \cite{azais}. With respect to  Proposition \ref{p:riceN},   we need in addition the  distribution of $ \nabla X(0), \nabla X(\rho e_1)$ to be non degenerated in $\R^{2N}$   
   (Condition  $(iii')$ of \cite{azais}). 
  From section \ref{s:gradient} we have
 \begin{align*}
\Var \left(\nabla X(0)\right)&=\Var \left(\nabla X(\rho e_1)\right)=\lambda_2 Id_N\\
\E \left(\nabla X(0) \nabla X(\rho e_1)^T\right) &= \mbox{diag}\Big( -2{\mathbf{r'}}(\rho^2) -4 \rho^2 {\mathbf{r''}}(\rho^2),-2{\mathbf{r'}}(\rho^2),\ldots, -2{\mathbf{r'}}(\rho^2) \Big),
\end{align*}
which implies  that the $N$ vectors $\big(X_i(0),X_i(\rho e_1) \big)$, $i=1,\ldots,N$ are independent. It remains to check  that the correlation between $X_i(0)$ and $X_i(\rho e_1)$ cannot be 1. This is done by the Taylor expansions
\begin{align*}
-2{\mathbf{r'}}(\rho^2) &= -2{\mathbf{r'}}(0)- 2\rho^2{\mathbf{r''}}(0) + o(\rho^2)\\
-2{\mathbf{r'}}(\rho^2)-4 \rho^2 {\mathbf{r''}}(\rho^2) &= -2{\mathbf{r'}}(0) -6 \rho^2{\mathbf{r''}}(0) + o(\rho^2).
\end{align*}
Since by hypothesis $ 0<{\mathbf{r''}}(0) = \lambda_4/12 <+\infty$  the result follows.

 We consider now the case of Kac-Rice formula  for the  mean number of critical points with a given index.  We proceed  using Theorem 6.4 of \cite{azais} as in the preceding case.
   \end{proof}

\section{Proofs of lemmas} \label{s:proofs}

\subsection{Proof of Lemma \ref{l:moments}}
We start by proving \eqref{eq0:lemme1}.
The spectral moment  of order $2p$ has the following expressions
\begin{equation}\label{e:lamb}
\lambda_{2p}   =                                                                                                                                                                                                                                                                                                                                                                                                          \int _{\R^N} \langle e_1, \lambda\rangle^{2p}   dF(\lambda) = \int_0^{+\infty}   x^{2p} dG(x)\int_{S^{N-1}}  \big( \langle e_1, u\rangle
 \big)^{2p}d \sigma(u),
\end{equation}
where $F$ is the spectral measure; $\sigma$ is the uniform probability on the sphere $S^{N-1}$ and $G$  is a measure  deduced from $F$  by change in polar coordinates.
Using the fact that a standard normal variable  is the independent product  of a  uniform variable on the sphere 
 by a $\chi(N)$ distribution and the fact that those two  distributions have well known moments, we obtain
 $$
 \mu_{2p}:=
 \int_{S^{N-1}} \big(  \langle e_1, u\rangle \big)^{2p}
d \sigma(u) =  \frac{(2p-1) !! \Gamma(N/2)}{2^p\Gamma(N/2+p) }= \frac{\Gamma(N/2) \Gamma(1/2+p)}
{\sqrt{\pi}\Gamma(N/2+p) }.
$$
We have
\begin{equation*}
\mu_{2n}\mu_{2n-4}=\frac{(2n-4+N)(2n-1)}{(2n-2+N)(2n-3)}\mu_{2n-2}^2.
\end{equation*}
Using Cauchy-Schwarz inequality we conclude that
\begin{align*}
\lambda_{2n} \lambda_{2n-4} &= \mu_{2n}\mu_{2n-4} \int_0^{+\infty} x^{2n}  dG(x)\int_0^{+\infty} y^{2n-4}  dG(y) \\
&= K(n,N) \mu_{2n-2}^2 \int_0^{+\infty} x^{2n}  dG(x)\int_0^{+\infty} y^{2n-4}  dG(y) \\
 &\geq  K(n,N) \mu_{2n-2}^2
 \int_0^{+\infty} x^{2n-2}  dG(x)\int_0^{+\infty} y^{2n-2} dG(y)  = K(n,N) \lambda_{2n-2}^2
\end{align*}
giving \eqref{eq0:lemme1}. Now let us prove \eqref{eq2:lemme1}. By  \eqref{e:lamb},  see also \cite{schoenberg} and \cite{yaglom}, $\mathbf r(\|s-t\|^2)$ can be written as
\begin{align}\nonumber
\mathbf r(\|s-t\|^2)
&=\int_0^{+\infty}        dG(x)      \int_{S^{N-1}} \cos \Big(  \big( \langle  \|s-t\|e_1, xu\rangle
 \big) \Big) d \sigma(u) \\ 
&=\int_0^{+\infty}  
 \sum_{k=0}^{\infty}(-1)^k\frac{x^{2k}\left(\|s-t\|^2\right)^k}{2^{2k+(N-2)/2}k!\Gamma(k+1+(N-2)/2)}dG(x). \label{cov}
\end{align}
Thus setting   $\displaystyle |\beta|=\frac{|i|+|j|}{2}$ and $\displaystyle \beta_k=\frac{i_k+j_k}{2}$ for $k=1,\ldots,N$, we have
\begin{multline*}
\left. \frac {\partial^{|i|+|j|} \mathbf r(\|s-t\|^2)}{\partial s_1^{i_1}, \ldots,\partial s_N^{i_N}    \partial t_1^{j_1}, \ldots,\partial t_N^{j_N}   }\right|_{s=t}
=\int_0^{+\infty}\frac{(-1)^{|\beta|}x^{2|\beta|}}{2^{2|\beta|+(N-2)/2}|\beta|!\Gamma(|\beta|+1+(N-2)/2)}\\
\times \left. \frac {\partial^{|i|+|j|}\left(\|s-t\|^2\right)^{|\beta|}}
{\partial s_1^{i_1}, \ldots,\partial s_N^{i_N}    \partial t_1^{j_1}, \ldots,\partial t_N^{j_N}   }\right|_{s=t}dG(x).
\end{multline*}
By the multinomial formula
\begin{equation} \label{multinomial}
\left(\|s-t\|^2\right)^{|\beta|}=\sum_{k_1+\ldots+k_N=|\beta|}\frac{|\beta|!}{k_1!\times \ldots \times k_N!}\left((s_1-t_1)^2\right)^{k_1}\times \ldots \times \left((s_N-t_N)^2\right)^{k_N}.
\end{equation}
Thus
$$
\left. \frac {\partial^{|i|+|j|}\left(\|s-t\|^2\right)^{|\beta|}}
{\partial s_1^{i_1}, \ldots,\partial s_N^{i_N}    \partial t_1^{j_1}, \ldots,\partial t_N^{j_N}   }\right|_{s=t}=(-1)^{|j|}|\beta|!\prod_{k=1}^N\frac{(2\beta_k)!}{\beta_k!},
$$
since the only term in the sum in the right side of \eqref{multinomial} whose
derivative is not null is the one for which $\displaystyle k_1=\frac{i_1+j_1}{2},\ldots,k_N=\frac{i_N+j_N}{2}$. 
We obtain
\begin{multline*}
\left. \frac {\partial^{|i|+|j|} \mathbf r(\|s-t\|^2)}{\partial s_1^{i_1}, \ldots,\partial s_N^{i_N}    \partial t_1^{j_1}, \ldots,\partial t_N^{j_N}   }\right|_{s=t}
=\frac{(-1)^{|\beta|}(-1)^{|j|}|\beta|!}{2^{2|\beta|+(N-2)/2}|\beta|!\Gamma(|\beta|+1+(N-2)/2)}\\
\times\prod_{k=1}^N\frac{(2\beta_k)!}{\beta_k!}\int_0^{+\infty} x^{2|\beta|}dG(x).
\end{multline*}
Furthermore, using \eqref{cov}, we have
$$
\mathbf r^{(|\beta|)}(0)=\frac{(-1)^{|\beta|}|\beta|!}{2^{2|\beta|+(N-2)/2}|\beta|!\Gamma(|\beta|+1+(N-2)/2)}\int_0^{+\infty}x^{2|\beta|}dG(x).
$$
That gives the first equality in \eqref{eq2:lemme1}. The second equality is obtained using \eqref{spectralmoments}.

\subsection{Proof of Lemma \ref{lemmeconditionaldistribution}}
  
 \noindent We give the steps for the computation of the conditional variance of $ \nabla^2X(0)$,  $\nabla^2X(t)$ given $\nabla X(0)=\nabla X(t)=0$ with $t=\rho e_1$. Some tedious but easy calculations are not detailed.
$\mathbf r(\rho^2)$, $\mathbf r'(\rho^2)$, $\mathbf r''(\rho^2)$, $\mathbf r'''(\rho^2)$ and $\mathbf r^{(4)}(\rho^2)$ are denoted  by $\mathbf r$, $\mathbf r'$, $\mathbf r''$, $\mathbf r'''$ and $\mathbf r^{(4)}$ for short. 
The results below extend Lemma \ref{l:moments} to two times. 
\subsubsection{Unconditional distribution}

\paragraph{Gradient} \label{s:gradient}
$$
 \Var(\nabla X(t))=\Var(\nabla X(0)) = -2\mathbf r '(0) Id_N,
$$
$$
 \Cov(X_1(0),X_1(t)) = -2\mathbf r '-4\rho^2 \mathbf r '' \mbox{ and }\Cov(X_i(0),X_i(t)) = -2\mathbf r ' \mbox{ for }  i\neq 1 .
$$
Any other covariance is zero.
\paragraph{Hessian} \label{s:hessien}
\noindent We define $X''_d(t)$ as $(X_{11}(t),\ldots,X_{NN}(t))$ and $X''_u(t)$ as 
$\{ X_{ij}(t), 1\leq i<j\leq N\}$. These two vectors are independent and 
\begin{align*}
 \Var(X''_d(0))  &= 4 \mathbf r ''(0)  ( 2 Id_N + J_{N,N}),\\
 \Cov(X''_d(0),X''_d(t)) &= 
   \left(\begin{array}{cc} 12  \mathbf r ''  + 48 \rho^2 \mathbf r ''' + 16 \rho^4 r^{(4)}&  (4\mathbf r '' + 8 \rho^2 \mathbf r ''') J_{1,(N-1)} \\
  (4\mathbf r '' + 8 \rho^2 \mathbf r ''') J_{(N-1),1}
  &    4 \mathbf r ''( 2 Id_{N-1}+ J_{(N-1),(N-1)})   \end{array}\right),
  \\
\Var(X''_u(0))  &  =  4\mathbf r ''(0) Id_{N(N-1)/2},\\
 \Cov(X''_u(0),X''_u(t)) &= \mbox{diag} \left((4 \mathbf r '' + 8 \rho^2 \mathbf r ''')Id_{N-1}, 4 \mathbf r '' Id_{(N-1)(N-2)/2}\right).
 \end{align*}
 
\paragraph{Relation between gradient and Hessian} \label{s:gradient_hessian}

\begin{align*}
&\E\big(X_{11}(0) \big/ X_{1}(0), X_1(t)\big)\\
 &=\frac{1}{4\mathbf r '(0)^2-(2\mathbf r '+4\rho^2\mathbf r '')^2}  (0,12 \rho \mathbf r '' +8\rho^3 \mathbf r ''') \left(\begin{array}{cc}-2\mathbf r '(0) & 2\mathbf r '+ 4\rho^2\mathbf r '' \\2\mathbf r '+4\rho^2\mathbf r '' & -2\mathbf r '(0)\end{array}\right)  \left(\begin{array}{c}X_1(0) \\X_1(t)\end{array}\right) \\
 &=(12 \rho \mathbf r '' +8\rho^3 \mathbf r ''') \frac{(2\mathbf r '+4\rho^2\mathbf r '')X_1(0) -2\mathbf r '(0)X_1(t)}{4\mathbf r '(0)^2-(2\mathbf r '+4\rho^2\mathbf r '')^2}\\
 &=: (12 \rho \mathbf r '' +8\rho^3 \mathbf r ''') K_1(t).
\end{align*}
\noindent For $i\neq 1$
\begin{align*}
\E\big(X_{1i}(0) \big/ X_i(0), X_i(t)\big)
 & =\frac{1}{4\mathbf r '(0)^2-4(\mathbf r ')^2} (0,4\rho \mathbf r '' ) \left(\begin{array}{cc}-2\mathbf r '(0) & 2\mathbf r ' \\2\mathbf r ' & -2\mathbf r '(0)\end{array}\right)  \left(\begin{array}{c}X_i(0) \\X_i(t)\end{array}\right) \\
 & = 4 \rho \mathbf r ''\frac{2\mathbf r 'X_i(0)-2\mathbf r '(0)X_i(t)}{4\mathbf r '(0)^2-4(\mathbf r ')^2}=: 4 \rho \mathbf r '' K_i(t).\\
\E\big(X_{ii}(0) \big/ X_1(0), X_1(t)\big) &= 4\rho \mathbf r '' K_1(t). 
\end{align*}
We have equivalent formulas, reversing time, for example: 
\begin{align*}
\E(X_{11}(t)/X_1(t),X_1(0)) &= (12 \rho \mathbf r '' +8\rho^3 \mathbf r ''') \frac{(-(2\mathbf r '+4\rho^2\mathbf r '')X_1(t) +2\mathbf r '(0)X_1(0))}{4(\mathbf r '(0))^2-(2\mathbf r '+4\rho^2\mathbf r '')^2}\\
& =:(12 \rho \mathbf r '' +8\rho^3 \mathbf r ''')  \bar K_1(t). 
\end{align*}
Any other case corresponds to independence between gradient and Hessian.  
   
\subsubsection{Conditional distribution}

\noindent Since the conditional variance-covariance matrix   is equal to the unconditional 
 variance-covariance matrix diminished by the variance of the conditional expectation,
  we compute  this last term only.\\
\noindent Let us consider  for example
 the two terms $X_{11} (0)$ and $X_{11} (t)$ given $\nabla X(0)=\nabla X(t)=0$. The $2\times 2$ matrix to subtract is $\displaystyle (12 \rho \mathbf r '' +8\rho^3 \mathbf r ''')^2\Var(K_1(t),\bar K_1(t))$ with
 \begin{equation*}
 \Var(K_1(t),\bar K_1(t))=\frac{1}{4(\mathbf r '(0))^2-(2\mathbf r '+4\rho^2\mathbf r '')^2}\left(\begin{array}{cc}
   -2\mathbf r '(0) &   -(2\mathbf r '+4\rho^2\mathbf r '')   \\  
-(2\mathbf r '+4\rho^2\mathbf r '')  & -2\mathbf r '(0)
   \end{array}\right).
\end{equation*}
In the same way  for $i \neq 1$ :
 \begin{equation*}
\Var(K_i(t);\bar K_i(t)) =\frac{1}{4\mathbf r '(0)^2-4(\mathbf r ')^2}\left(\begin{array}{cc}
   -2\mathbf r '(0) &  -2\mathbf r ' \\  
 -2\mathbf r ' & -2\mathbf r '(0)
   \end{array}\right).
   \end{equation*}
Giving the extra term to substract to get Lemma \ref{lemmeconditionaldistribution}.

 \subsection{Proof of Lemma \ref{lemme2}}

First note that \eqref{zero1} and \eqref{zero} are deduced from Lemma \ref{lemmeconditionaldistribution}.

\subsubsection{Proof of \eqref{eqvarxi1j}, \eqref{eqvar2} and \eqref{determinant1}}
\noindent We now consider  the case of $\xi_{11}(t)$. Using Trick \ref{trick1} we have
\begin{align*}
\Var\left(\xi_{11}(t)\right)=&\frac{\det (\Var(X_1(0),\ldots,X_N(0),X_{11}(0),X_1(t),\ldots,X_N(t)))}{\det\left(\Var\left(\nabla X(0),\nabla X(t)\right)\right)}\\
\simeq &\det (\Var(X_1(0),\ldots,X_N(0),X_{11}(0),X_1(t)-X_1(0)-\rho X_{11}(0),\\
&X_2(t)-X_2(0),\ldots,X_N(t)-X_N(0)))/\det\left(\Var\left(\nabla X(0),\nabla X(t)\right)\right)\\
\simeq &\frac{\rho^2 3^{N-1}}{4\lambda_2^N\lambda_4^N} \det (\Var (X_1(0),\ldots,X_N(0),X_{11}(0),\ldots,X_{1N}(0),X_{111}(0) )).
\end{align*}
By  Lemma \ref{lemme1} we obtain $\displaystyle \Var\left(\xi_{11}(t)\right) \simeq \frac{\rho^2}{4}\frac{(\lambda_2\lambda_6-\lambda_4^2)}{\lambda_2}$.  The other variances in \eqref{eqvarxi1j} and \eqref{eqvar2}
and the determinant \eqref{determinant1}
are obtained in the same way. 
\subsubsection{Proof of \eqref{eqvarxijk}, \eqref{eqcov2} and \eqref{eqvarxi11} }   ~ \newline
$\bullet$ We consider $\Cov (\xi_{jj}(t),\xi_{kk}(t))$ for $j\neq k$ and $j,k\neq 1$. Since
$$
\Cov (\xi_{jj}(t),\xi_{kk}(t))=\frac{1}{2}\left(\Var \left(\xi_{jj}(t)+\xi_{kk}(t)\right)-\Var (\xi_{jj}(t))-\Var (\xi_{kk}(t))\right)
$$
and since, by \eqref{eqvarxi1j}, $\displaystyle \Var(\xi_{jj}(t))=\Var (\xi_{kk}(t))=\frac{8\lambda_4}{9}$, we just need to compute 
$$
\Var\left(\xi_{jj}(t)+\xi_{kk}(t)\right)=\frac{\det (\Var(X_{jj}(t)+X_{kk}(t),X_1(0),\ldots,X_N(0),X_1(t),\ldots,X_N(t)))}{\det\left(\Var\left(\nabla X(0),\nabla X(t)\right)\right)}.
$$
Using Trick \ref{trick1} we obtain
\begin{align*}
&\Var\left(\xi_{jj}(t)+\xi_{kk}(t)\right)\\
\simeq& \frac{\det (\Var(X_{jj}(0)+X_{kk}(0),X_1(0),\ldots,X_N(0),X_{11}(0),X_{12}(0),\ldots,X_{1N}(0)))}
{\det\left(\Var\left(X_1(0),\ldots,X_N(0),X_{11}(0),X_{12}(0),\ldots,X_{1N}(0)\right)\right)}.
\end{align*}
By  Lemma \ref{lemme1} we have
\begin{align*}
&\det (\Var(X_{jj}(0)+X_{kk}(0),X_1(0),\ldots,X_N(0),X_{11}(0),X_{12}(0),\ldots,X_{1N}(0)))\\
=&\left(\frac{\lambda_4}{3}\right)^{N-1}\lambda_2^N\det (\Var (X_{11}(0),X_{jj}(0)+X_{kk}(0))), 
\end{align*}
$$
\det (\Var (X_{11}(0),X_{jj}(0)+X_{kk}(0)))=\frac{20\lambda_4^2}{9}
$$
and
$$
\det\left(\Var\left(X_1(0),\ldots,X_N(0),X_{11}(0),X_{12}(0),\ldots,X_{1N}(0)\right)\right)=\frac{\lambda_2^N\lambda_4^N}{3^{N-1}}.
$$
Finally we obtain $\displaystyle \Var(\xi_{jj}(t)+\xi_{kk}(t)) \simeq \frac{20\lambda_4}{9}$, implying $ \displaystyle \Cov (\xi_{jj}(t),\xi_{kk}(t))\simeq \frac{2\lambda_4}{9}$.
In the same way we obtain$ \displaystyle \Cov (\xi_{jj}(0),\xi_{kk}(t))\simeq \frac{2\lambda_4}{9}$, $ \displaystyle \Cov (\xi_{jk}(0),\xi_{jk}(t))\simeq \Var(\xi_{jk}(t))$ for $j\neq k$
and $ \displaystyle \Cov (\xi_{1i}(0),\xi_{1i}(t))\simeq -\Var(\xi_{1i}(t))$ $\forall i\in \{1,\ldots,N\}$.\\
\\
\noindent $\bullet$ Now let us prove that $\displaystyle \Cov(\xi_{11}(t),\xi_{jj}(t))  =  \rho^2\frac{11\lambda_2\lambda_6-15\lambda_4^2}{180\lambda_2} + o(\rho^2)$ for $j\neq 1$. We have,
(see Lemma \ref{lemme1} and Sections \ref{s:hessien} and \ref{s:gradient_hessian})
\begin{align*}
\Cov(X_{jj}(0),X_i(0))&=0,\mbox{ }\forall i,j\in \{1,\ldots,N\}\\
\Cov(X_{jj}(0),X_i(t))=-\Cov(X_{jj}(t),X_i(0))&=4\rho \mathbf r ''(\rho^2)\delta_{1i}\mbox{ for }j\neq i,\\
\Cov(X_{ii}(0),X_i(t))=-\Cov(X_{ii}(t),X_i(0))&=(12\rho \mathbf r ''(\rho^2)+8\rho^3 \mathbf r '''(\rho^2))\delta_{1i} \mbox{ and}\\
\Cov(X_{11}(0),X_{jj}(t))&=4\mathbf r^{''}(\rho^2)+8\rho^2\mathbf r^{'''}(\rho^2).
\end{align*}
With Section \ref{s:gradient} we deduce
\begin{align*}
&\Var(\xi_{jj}(0))=\frac{\det (\Var(X_1(0),\ldots,X_N(0),X_1(t),\ldots,X_N(t),X_{jj}(0)))}{\det\left(\Var\left(\nabla X(0),\nabla X(t)\right)\right)}\\
&=\frac{\det (\Var(\nabla_{-(1)}X(0),\nabla_{-(1)}X(t)))\times \det (\Var(X_1(0),X_1(t),X_{jj}(0)))}{\det\left(\Var\left(\nabla X(0),\nabla X(t)\right)\right)},
\end{align*}
 where $\nabla_{-(1)}X$ denotes the gradient without its first coordinate.
Similarly
$$
\Var(\xi_{11}(0))=\frac{\det (\Var(\nabla_{-(1)}X(0),\nabla_{-(1)}X(t)))\times \det (\Var(X_1(0),X_1(t),X_{11}(0)))}{\det\left(\Var\left(\nabla X(0),\nabla X(t)\right)\right)}
$$
and
\begin{align*}
&\Var(\xi_{11}(0)+\xi_{jj}(0))=\\
&\frac{\det (\Var(\nabla_{-(1)}X(0),\nabla_{-(1)}X(t)))\times \det (\Var(X_1(0),X_1(t),X_{11}(0)+X_{jj}(0)))}{\det\left(\Var\left(\nabla X(0),\nabla X(t)\right)\right)}.
\end{align*}
Using Lemma \ref{lemme1} we can verify that
\begin{align*}
\det (\Var(X_1(0),X_1(t),X_{11}(0)+X_{jj}(0)))&=\frac{8}{3}\lambda_4\lambda_2^2-\lambda_2(\alpha_1+\alpha_2)^2-\frac{8}{3}\lambda_4\beta_1^2\\
\det (\Var(X_1(0),X_1(t),X_{11}(0)))&=\lambda_4\lambda_2^2-\lambda_2 \alpha_1^2-\lambda_4\beta_1^2\\
\det (\Var(X_1(0),X_1(t),X_{jj}(0)))&=\lambda_4\lambda_2^2-\lambda_2 \alpha_2^2-\lambda_4\beta_1^2,
\end{align*}
where $\beta_1=-2\mathbf r '(\rho^2)-4\rho^2\mathbf r ''(\rho^2)$, $\alpha_1=12\rho \mathbf r ''(\rho^2)+8\rho^3 \mathbf r '''(\rho^2)$ and $\alpha_2=4\rho \mathbf r ''(\rho^2)$.
We deduce
$$
\Cov(\xi_{11}(0),\xi_{jj}(0))=\frac{\det (\Var(\nabla_{-(1)}X(0),\nabla_{-(1)}X(t)))\times \left(\displaystyle \frac{1}{3}\lambda_4\lambda_2^2-\lambda_2\alpha_1\alpha_2-\frac{1}{3}\lambda_4\beta_1^2\right)}{\det\left(\Var\left(\nabla X(0),\nabla X(t)\right)\right)}.
$$
We can check that
$$
\displaystyle \frac{1}{3}\lambda_4\lambda_2^2-\lambda_2\alpha_1\alpha_2-\frac{1}{3}\lambda_4\beta_1^2 \simeq \rho^4 \left(\frac{11}{180}\lambda_2\lambda_4\lambda_6-\frac{\lambda_4^3}{12}\right). 
$$
We deduce, using Trick \ref{trick1} and Lemma \ref{lemme1}, that
$$
\Cov(\xi_{11}(0),\xi_{jj}(0))= \rho^2\frac{11\lambda_2\lambda_6-15\lambda_4^2}{180\lambda_2} +o(\rho^2).
$$
\noindent In the same way we prove $\displaystyle \Cov(\xi_{11}(0),\xi_{jj}(t)) = \rho^2\frac{15\lambda_4^2-7\lambda_2\lambda_6}{180\lambda_2} +o(\rho^2)$.


\begin{thebibliography}{4}

\bibitem{adler}
\textsc{Adler, R. J.} (1981).
\textit{The Geometry of Random Fields}.
Wiley, New York.


\bibitem{adlertaylor}
\textsc{Adler, R. J.} and \textsc{Taylor, J.} (2007).
\textit{Random Fields and Geometry}.
Springer, New York.



\bibitem{benarous}
\textsc{Auffinger, A.} and \textsc{Ben Arous, G.} (2013).
Complexity of random smooth functions on the high-dimensional sphere.
\textit{Ann. Probab.} 
\textbf{41(6)} 4214--4247.

\bibitem{benarous3}
\textsc{Auffinger, A.}, \textsc{Ben Arous, G.} and \textsc{\v{C}ern\'y, J.} (2013).
Random matrices and complexity of spin glasses.
\textit{Comm. Pure Appl. Math.} 
\textbf{66(2)} 165--201.

\bibitem{azais1}
\textsc{Aza\"is, J. M.}, \textsc{Cierco-Ayrolles, C.} and \textsc{Croquette, A.} (1999).
Bounds and asymptotic expansions for the distribution of the maximum of a smooth stationary Gaussian process.
\textit{ESAIM: Probab. Statist.}
\textbf{3} 107--129.


\bibitem{Al}
\textsc{Aza\"is, J. M.} and \textsc{Le\'on, J. R.} (2019).
Necessary and sufficient conditions for the finiteness of the second moment of the measure of level sets.
\textit{arXiv:1905.12342}

\bibitem{Alo}
\textsc{Aza\"is, J. M.}, \textsc{Le\'on, J. R.} and \textsc{Ortega, J.} (2005).
Geometrical characteristic of Gaussian sea waves.
\textit{J. Appl. Probab.}
\textbf{42} 407--425.

\bibitem{AW3}
\textsc{Aza\"is, J. M.} and \textsc{Wschebor, M.} (2008)
A general expression for the distribution of the maximum of a Gaussian field and the approximation of the tail.
\textit{Stoch. Process. Appl.}
\textbf{118} 1190--1218.

\bibitem{azais}
\textsc{Aza\"is, J. M.} and \textsc{Wschebor, M.} (2009).
\textit{Level Sets and Extrema of Random Processes and Fields}.
John Wiley \& Sons, Inc., Hoboken, New Jersey.


\bibitem{AW2}
\textsc{Aza\"is, J. M.} and \textsc{Wschebor, M.} (2010)
Erratum to: A general expression for the distribution of the maximum of a Gaussian field and the approximation of the tail [Stochastic Process. Appl. 118 (7)(2008) 1190–1218].
\textit{Stoch. Process. Appl.}
\textbf{120} 2100--2101.



\bibitem{BCW}
\textsc{Beliaev, D.}, \textsc{Cammarota, V.} and \textsc{Wigman, I.} (2019a).
Two Point Function for Critical Points of a Random Plane Wave.
\textit{International Mathematics Research Notices}
\textbf{9} 2661--2689.

\bibitem{BCW2}
\textsc{Beliaev, D.}, \textsc{Cammarota, V.} and \textsc{Wigman, I.} (2019b).
No repulsion between critical points for planar Gaussian random fields.
\textit{arXiv:1704.04943}




\bibitem{benarous2}
\textsc{Ben Arous, G.}, \textsc{Mei, S.}, \textsc{Montanari, A.} and \textsc{Nica, M.} (2019).
The Landscape of the Spiked Tensor Model.
\textit{Communications on Pure and Applied Mathematics} 
\textit{https://doi.org/10.1002/cpa.21861 }



\bibitem{berry}
\textsc{Berry, M. V.} and \textsc{Dennis, M. R.} (2000).
Phase singularities in isotropic random waves.
\textit{Proc. R. Soc. Lond. A}
\textbf{456} 2059--2079.

\bibitem{bruijn}
\textsc{de Bruijn, N. G.} (1955).
On some multiple integrals involving determinants.
\textit{J. Indian Math. Soc.}
\textbf{19}, 133--151.






\bibitem{cheng3}
\textsc{Cheng, D.} and \textsc{Schwartzman, A.} (2017).
Multiple testing of local maxima for detection of peaks in random fields.
\textit{Ann. Stat.}
\textbf{45} 529--556.



\bibitem{cheng2}
\textsc{Cheng, D.} and \textsc{Schwartzman, A.} (2018).
Expected number and height distribution of critical points of smooth isotropic Gaussian random fields.
\textit{Bernoulli}
\textbf{24(4B)} 3422--3446.



\bibitem{cheng}
\textsc{Cheng, D.} and \textsc{Xiao, Y.} (2016).
The mean Euler characteristic and excursion probability of Gaussian random fields with stationary increments.
\textit{Ann. Appl. Probab.}
\textbf{26(2)} 722--759.



\bibitem{ginsbourger}
\textsc{Chevalier, C.} and \textsc{Ginsbourger, D.} (2013). 
Fast computation of the multi-points expected improvement with applications in batch selection. 
\textit{International Conference on Learning and Intelligent Optimization}.
Springer, Berlin, Heidelberg.

\bibitem{chiani}
\textsc{Chiani, M.} (2014). 
Distribution of the largest eigenvalue for real Wishart and Gaussian random matrices and a simple approximation for the Tracy-Widom distribution. 
\textit{J. Multivariate Anal.}
\textbf{129} 69--81.


\bibitem{CL}
\textsc{Cramér, H.} and \textsc{Leadbetter, M. R.} (1967).
\textit{Stationary and Related Stochastic Processes}.
Wiley, New York.

\bibitem{estrade}  
\textsc{Estrade, A.} and \textsc{Fournier, J.} (2016).
Number of critical points of a Gaussian random field: Condition for a finite variance.
\textit{Statist. Probab. Lett.}
\textbf{118} 94--99.




%

\bibitem{hough}
\textsc{Hough, J. B.}, \textsc{Krishnapur, M.}, \textsc{Peres, Y.} and \textsc{Virag, B.} (2006).
Determinantal processes and independence. 
\textit{Probab. Surveys}
\textbf{3} 206-229.

\bibitem{longuet}
\textsc{Longuet-Higgins, M. S.} (1957).
The statistical analysis of a random moving surface. 
\textit{Phil. Trans. R. Soc. London A}
\textbf{249} 321-387.


\bibitem{mehta}
\textsc{Mehta, M. L.} (2004).
\textit{Random Matrices}.
3rd ed. Academic Press, San Diego, CA.


\bibitem{mehta1}
\textsc{Mehta, M. L.} and \textsc{Normand, J. M.} (2001). 
Moments of the characteristic polynomial in the three ensembles of random matrices.
\textit{J. Phys. A: Math. Gen.}
\textbf{34} 4627--4639.


\bibitem{nicolaescu}
\textsc{Nicolaescu, L. I.} (2017).
A CLT concerning critical points of random functions on a Euclidean space.
\textit{Stoch. Process. Appl.},
\textbf{127(10)} 3412--3446.



\bibitem{rychlik}
\textsc{Podg\'orski, K.}, \textsc{Rychlik, I.} and \textsc{Sj\"o, E.} (2000).
Statistics for velocities of Gaussian waves.
\textit{Internat. J. Offshore Polar Eng.}
\textbf{10} 91--98.


\bibitem{ros}
\textsc{Ros, V.}, \textsc{Ben Arous, G.}, \textsc{Biroli, G.} and \textsc{Cammarota, C.} (2019).
Complex energy landscapes in spiked-tensor and simple glassy models: Ruggedness, arrangements of local minima, and phase transitions.
\textit{Phys. Rev. X}
\textbf{9}, 011003.


\bibitem{searle}
\textsc{Searle, S. R.} (1971). 
\textit{Linear Models}.
John Wiley \& Sons.


\bibitem{schoenberg}
\textsc{Schoenberg, I. J.} (1938).
Metric spaces and completely monotone functions.
\textit{Ann. of Math.}
\textbf{39} 811--841.


\bibitem{taylor}
\textsc{Taylor, J. E.} and \textsc{Worsley, K. J.} (2007).
Detecting sparse signals in random fields, with an application to brain mapping.
\textit{J. Am. Statist. Assoc.}
\textbf{102} 913--928.

\bibitem{worsley1}
\textsc{Worsley, K. J.}, \textsc{Marrett, S.}, \textsc{Neelin, P.} and \textsc{Evans, A. C.} (1996).
Searching scale space for activation in PET images.
\textit{Human Brain Mapping}
\textbf{4} 74--90.

\bibitem{worsley2}
\textsc{Worsley, K. J.}, \textsc{Marrett, S.}, \textsc{Neelin, P.}, \textsc{Vandal, A. C.}, \textsc{Friston, K.J.} and \textsc{Evans, A. C.} (1996).
A unified statistical approach for determining significant signals in images of cerebral activation.
\textit{Human Brain Mapping}
\textbf{4} 58--73.

\bibitem{yaglom}
\textsc{Yaglom, A. M.} (1957).
Some classes of random fields in $n$-dimensional space, related to stationary random processes.
\textit{Theor. Probability Appl.}
\textbf{2} 273--320.


\end{thebibliography}
\end{document}